\documentclass[aop,MSNbibl,seceqn,dvips]{arximspdf}
\usepackage{mathbh}

\doi{10.1214/13-AOP861} %kopijuoti is PTS
\volume{42}
\issue{1}
\pubyear{2014}
\firstpage{311}
\lastpage{353}

\makeatletter
\def\leftidx{}

\newcommand{\eqref}[1]{(\ref{#1})}
\newtheorem{theorem}{Theorem}[section]
\newtheorem{proposition}[theorem]{Proposition}
\newtheorem{lemma}[theorem]{Lemma}
\newtheorem{corollary}[theorem]{Corollary}
\newtheorem{fact}[theorem]{Fact}
\newproclaim{definition}[theorem]{Definition}
\newproclaim{remark}[theorem]{Remark}
\newproclaim{question}[theorem]{Question}
\newproclaim{notation}[theorem]{Notation}
\newproclaim{example}[theorem]{Example}

\renewcommand{\phi}{\varphi}

\newcommand{\D}{\mathrm{d}}

\newcommand{\ti}{\tilde}

\renewcommand{\equiv}{\Longleftrightarrow}
\newcommand{\equ}{$ \Longleftrightarrow$}
\newcommand{\imply}{ \Longrightarrow}
\newcommand{\impl}{ \Longrightarrow}
\newcommand{\imp}{$ \Longrightarrow$ }

\newcommand{\Atoms}{\operatorname{Atoms}}
\newcommand{\Down}{\operatorname{Down}}
\newcommand{\Up}{\operatorname{Up}}
\newcommand{\Cl}{\operatorname{Cl}}

\newcommand{\even}{\mathrm{even}}
\newcommand{\odd}{\mathrm{odd}}

\newcommand{\One}{\mathbh{1}}

\newcommand{\eps}{\varepsilon}
\newcommand{\Si}{\Sigma}
\newcommand{\ga}{\gamma}
\newcommand{\Om}{\Omega}
\newcommand{\al}{\alpha}
\newcommand{\be}{\beta}
\newcommand{\M}{\mathcal{ M}}
\newcommand{\F}{\mathcal{F}}
\newcommand{\A}{\mathcal{ A}}
\newcommand{\La}{\Lambda}

\newcommand{\Q}{\mathbf{Q}}

\newcommand{\Ex}{\mathbb{E}}
\newcommand{\R}{\mathbb R}

\makeatother

\begin{document}
\begin{frontmatter}

\title{Noise as a Boolean algebra of $\sigma$-fields}
\runtitle{Noise as a Boolean algebra of $\sigma$-fields}

\begin{aug}
\author{\fnms{Boris} \snm{Tsirelson}\corref{}\ead[label=e1]{tsirel@post.tau.ac.il}}% \thanksref{t1}}
\runauthor{B. Tsirelson}
\affiliation{Tel Aviv University}
\address{School of Mathematics\\
Tel Aviv University\\
Tel Aviv 69978\\
Israel\\
\printead{e1}}
\end{aug}

% HISTORY:
\received{\smonth{11} \syear{2011}}
\revised{\smonth{3} \syear{2013}}

% ABSTRACT
%
\begin{abstract}
A noise is a kind of homomorphism from a Boolean algebra of domains to
the lattice of $\sigma$-fields. Leaving aside the homomorphism we
examine its
image, a Boolean algebra of $\sigma$-fields. The largest extension of such
Boolean algebra of $\sigma$-fields, being well-defined always, is a complete
Boolean algebra if and only if the noise is classical, which answers
an old question of J.~Feldman.
\end{abstract}

% KEYWORDS
\begin{keyword}[class=AMS]
\kwd[Primary ]{60G99}
\kwd[; secondary ]{60A10}
\kwd{60G20}
\kwd{60G60}
\end{keyword}
\begin{keyword}
\kwd{Black noise}
\end{keyword}

\end{frontmatter}
\section*{Introduction}

The product of two measure spaces, widely known among
mathematicians, leads to the tensor product of the corresponding
Hilbert spaces~$ L_2 $. The less widely known product of an infinite
sequence of probability spaces leads to the so-called infinite tensor
product space. A continuous product of probability spaces, used in the
theory of noises, leads to a continuous tensor product of Hilbert
spaces, used in noncommutative dynamics. Remarkable parallelism and
fruitful interrelations between the two theories of continuous
products, commutative (probability) and noncommutative (operator
algebras) are noted \cite{TV,Ts02,Ts04}.

The classical theory, developed in the 20th century, deals with
independent increments (L\'evy processes) in the commutative case, and
quasi-free representations of canonical commutation relations (Fock
spaces) in the noncommutative case. These classical continuous
products are well understood, except for one condition of
classicality, whose sufficiency was conjectured by H.~Araki and
E.~J. Woods in 1966 (\cite{AW}, page~210), in the noncommutative case (still
open), and by J.~Feldman in 1971 (\cite{Fe}, Problem 1.9), in the
commutative case (now proved).

Araki and Woods note (\cite{AW}, pages~161--162), that lattices of von
Neumann algebras occur in quantum field theory and quantum statistical
mechanics; these algebras correspond to domains in space--time or
space; in most interesting cases they fail to be a Boolean algebra of
type I factors. As a first step toward an understanding of such
structures, Araki and Woods investigate ``factorizations,'' complete
Boolean algebras of type I factors, leaving aside their
relation to the domains in space(--time), and conjecture that all such
factorizations contain sufficiently many factorizable vectors.

Feldman defines ``factored probability spaces'' that are in fact
complete Boolean algebras of sub-$\sigma$-fields (corresponding to Borel
subsets of a parameter space, which does not really matter),
investigates them assuming sufficiently many ``decomposable
processes'' (basically the same as factorizable vectors) and asks
whether this assumption holds always, or not.

In both cases the authors failed to prove that the completeness of the
Boolean algebra implies classicality (via sufficiently many
factorizable vectors).

In both cases the authors did not find any nonclassical
factorizations, and did not formulate an appropriate framework for
these. This challenge in the noncommutative case was met in 1987 by
Powers \cite{Po} (``type III product system''), and in the
commutative case in 1998 by Vershik and myself \cite{TV} (``black
noise''). In both cases the framework was an incomplete Boolean
algebra indexed by one-dimensional intervals and their finite
unions. More interesting nonclassical noises were found soon (see the
survey \cite{Ts04}), but the first highly important example is given
recently by Schramm, Smirnov and Garban \cite{SS}---the
noise of percolation, a conformally invariant black noise over the
plane.

Being indexed by planar domains (whose needed regularity depends on
some properties of the noise), such a noise exceeds the limits of the
existing framework based on one-dimensional intervals. Abandoning the
intervals, it is natural to return to the Boolean algebras, leaving
aside (once again!) their relations to planar (or more general)
domains; this time, however, the Boolean algebra is generally
incomplete.

The present article provides a remake of the theory of noises, treated
here as Boolean algebras of $\sigma$-fields. Completeness of the Boolean
algebra implies classicality, which answers the question of Feldman.

The noncommutative case is still waiting for a similar treatment.

The author thanks the anonymous referee and the associate editor;
several examples and the whole Section~\ref{1f} are added on their
advices.

%s1 #&#
\section{Main results}
\label{sec1}
% \input{sect1}

%s1.1 #&#
\subsection{Definitions}
\label{1a}

Let $ (\Om,\F,P) $ be a probability space; that is, $ \Om$ is a set,
$ \F$ a $\sigma$-field (in other words, $\sigma$-algebra) of its subsets
(throughout, every $\sigma$-field is assumed to contain all null
sets), and $ P
$ a probability measure on $ (\Om,\F) $. We assume that $
L_2(\Om,\F,P) $ is separable. The set $ \La$ of all sub-$\sigma
$-fields of $
\F$ is partially ordered (by inclusion: $ x \le y $ means $ x \subset
y $ for $ x,y \in\La$), and is a lattice:
\[
x \wedge y = x \cap y, \qquad x \vee y = \sigma(x,y)\qquad \mbox{for } x,y \in\La;
\]
here $ \sigma(x,y) $ is the least $\sigma$-field containing both $ x
$ and $ y
$. (See \cite{DP} for basics about lattices and Boolean algebras.)
The greatest element $ 1_\La$ of $ \La$ is $ \F$; the smallest
element $ 0_\La$ is the trivial $\sigma$-field (only null sets and their
complements).

A subset $ B \subset\La$ is called a sublattice if $ x \wedge y, x
\vee y \in B $ for all $ x,y \in B $. The sublattice is called
distributive if $ x \wedge( y \vee z ) = ( x \wedge y ) \vee( x
\wedge z ) $ for all $ x,y,z \in B $.

Let $ B \subset\La$ be a distributive sublattice, $ 0_\La\in B $,
$ 1_\La\in B $. An element $ x $ of $ B $ is called complemented (in
$ B $), if $ x \wedge y = 0_\La$, $ x \vee y = 1_\La$ for some
(necessarily unique) $ y \in B $; in this case one says that $ y $ is
the complement of $ x $, and writes $ y = x' $.

%
%de1.1 #&#
\begin{definition}\label{1a1}
A \emph{noise-type Boolean algebra} is a distributive sublattice $ B
\subset\La$ such that $ 0_\La\in B $, $ 1_\La\in B $, all elements
of $ B $ are complemented (in $ B $), and for every $ x \in B $ the
$\sigma$-fields $ x, x' $ are independent [i.e., $ P ( X \cap Y ) = P(X)
P(Y) $ for all $ X \in x $, $ Y \in y $].
\end{definition}

From now on $ B \subset\La$ is a noise-type Boolean algebra.

%
%de1.2 #&#
\begin{definition}\label{1a2}
The \emph{first chaos space} $ H^{(1)}(B) $ is a (closed linear)
subspace of the Hilbert space $ H = L_2(\Om,\F,P) $ consisting of all
$ f \in H $ such that
\[
f = \mathbb{E}(f |x) + \mathbb{E}\bigl(f |x'\bigr) \qquad\mbox{for all }
x \in B.
\]
\end{definition}

Here $ \mathbb{E}(\cdot |x)$ is the conditional expectation, that
is, the
orthogonal projection onto the subspace $ H_x $ of all $x$-measurable
elements of $ H $.

%
%de1.3 #&#
\begin{definition}\label{1a3}
(a) $ B $ is called \emph{classical} if the first chaos space
generates the whole $\sigma$-field $ \F$.

(b) $ B $ is called \emph{black} if the first chaos space contains
only $ 0 $ (but $ 0_\La\ne1_\La$).
\end{definition}

The lattice $ \La$ is complete; that is, every subset $ X \subset\La
$ has an infimum and a supremum,
\[
\inf X = \bigcap_{x\in X} x, \qquad\sup X = \sigma\biggl(
\bigcup_{x\in X} x \biggr).
\]
A noise-type Boolean algebra $ B $ is called complete if
\[
(\inf X) \in B \quad\mbox{and}\quad (\sup X) \in B\qquad \mbox{for every } X \subset B.
\]

%s1.2 #&#
\subsection{The simplest nonclassical example}
\label{1b}

Let $ \Om= \{-1,1\}^\infty$ (all infinite sequences of $ \pm1 $)
with the product measure $ \mu^\infty$ where $
\mu(\{-1\})=\mu(\{1\})=1/2 $. The coordinate projections $ \xi
_n\dvtx\Om
\to\{-1,1\} $, $ \xi_n(s_1,s_2,\ldots) = s_n $, treated as random
variables, are independent random signs. The products $ \xi_1 \xi_2,
\xi_2 \xi_3,\break  \xi_3 \xi_4, \ldots$ are also independent random
signs.\vadjust{\goodbreak}

We introduce $\sigma$-fields
\[
x_n = \sigma(\xi_n,\xi_{n+1},\ldots)\quad \mbox{and}\quad
y_n = \sigma(\xi_n \xi_{n+1})\qquad \mbox{for } n=1,2,
\ldots.
\]
Then
\begin{eqnarray*}
&1_\La= x_1 \ge x_2 \ge\cdots;&
\\
&y_n \le x_n;&
\\
&y_1,\ldots,y_n,x_{n+1} \mbox{ are independent};&
\\
&y_n \vee x_{n+1} = x_n.&
\end{eqnarray*}
The independent $\sigma$-fields $ y_1,\ldots,y_n,x_{n+1} $ are atoms
of a finite
noise-type Boolean algebra $ B_n $ (containing $ 2^{n+1} $ elements),
and $ B_n \subset B_{n+1} $. The union
\[
B = B_1 \cup B_2 \cup\cdots
\]
is an infinite noise-type Boolean algebra. As a Boolean algebra, $B$
is isomorphic to the finite/cofinite Boolean algebra, that is, the
algebra of all finite subsets of $ \{1,2,\ldots\} $ and their
complements; $ x_n \in B $ corresponds to the cofinite set
$\{n,n+1,\ldots\}$, while $ y_n \in B $ corresponds to the
single-element set $\{n\}$. The first chaos space $ H^{(1)}(B) =
H^{(1)}(B_1) \cap H^{(1)}(B_2) \cap\cdots$ consists of linear
combinations
\[
c_1 \xi_1 \xi_2 + c_2
\xi_2 \xi_3 + c_3 \xi_3
\xi_4 + \cdots
\]
for all $ c_1,c_2,\ldots\in\R$ such that $ c_1^2 + c_2^2 + \cdots<
\infty$. It is not $ \{0\} $, which shows that $ B $ is not black. On
the other hand, all elements of $ H^{(1)}(B) $ are invariant under the
measure preserving transformation $ (s_1,s_2,\ldots) \mapsto
(-s_1,-s_2,\ldots) $; therefore $ \sigma(H^{(1)}(B)) $ is not the
whole $
1_\La$, which shows that $ B $ is not classical.

The complement $ x'_n $ of $ x_n $ in $ B $ is $ y_1 \vee\cdots\vee
y_{n-1} = \sigma( \xi_1 \xi_2, \xi_2 \xi_3, \ldots, \xi_{n-1} \xi
_n )
$. Clearly, $ x_n \downarrow0_\La$ (i.e., $ \inf_n x_n = 0_\La
$). Strangely, the relation $ x'_n \uparrow1_\La$ fails; $ x'_n
\uparrow\sup_n y_n = \sigma( \xi_1 \xi_2, \xi_2 \xi_3, \ldots)
\ne
1_\La$. ``The phenomenon \ldots tripped up even Kolmogorov and
Wiener'' \cite{Wi}, Section~4.12.

This example goes back to an unpublished dissertation of Vershik
\cite{Ve73}. According to Emery and Schachermayer (\cite{ES}, page~291), it
is a paradigmatic example,
well known in ergodic theory, independently discovered by several
authors. See also \cite{Wi}, Section~4.12, \cite{Ts04}, Section~1b.

%s1.3 #&#
\subsection{On Feldman's question}
\label{1c}

%
%th1.4 #&#
\begin{theorem}\label{1c1}
If a noise-type Boolean algebra is complete, then it is classical.
\end{theorem}

%
%th1.5 #&#
\begin{theorem}\label{1c2}
The following conditions on a noise-type Boolean algebra $ B $ are
equivalent:
\begin{longlist}[(a)]
\item[(a)] $ B $ is classical;

\item[(b)] there exists a complete noise-type Boolean algebra $ \hat
B $ such
that $ B \subset\hat B $;

\item[(c)] $ ( \sup_n x_n ) \vee( \inf_n x'_n ) = 1_\La$ for all
$ x_n
\in B $ such that $ x_1 \le x_2 \le\cdots$.\vadjust{\goodbreak}
\end{longlist}
\end{theorem}

See also Theorem \ref{7d2} for another important condition of classicality.

%s1.4 #&#
\subsection{On completion}
\label{1d}

Bad news: a noise-type Boolean algebra cannot be extended to a
complete one unless it is classical. (See Theorem \ref{1c2}. True,
every Boolean algebra admits a completion \cite{Ha}, Section~21, but not
within~$\La$.)

Good news: an appropriate notion of completion exists and is described
below (Definition \ref{1d3}).

The lower limit
\[
\liminf_n x_n = \sup_n
\inf_k x_{n+k}
\]
is well defined for arbitrary $ x_1,x_2,\ldots\in\La$. (The upper
limit is defined similarly.)

%
%th1.6 #&#
\begin{theorem}\label{1d1}
Let $ B $ be a noise-type Boolean algebra and
\[
\Cl(B) = \Bigl\{ \liminf_n x_n\dvtx x_1,x_2,\ldots\in B \Bigr\}
\]
(the set of lower limits of all sequences of elements of $ B $). Then:
\begin{longlist}[(a)]
\item[(a)] $ (\inf_n x_n) \in\Cl(B) $ whenever $ x_1,x_2,\ldots
\in\Cl(B) $;

\item[(b)] $ (\sup_n x_n) \in\Cl(B) $ whenever $ x_1,x_2,\ldots
\in\Cl(B) $, $
x_1 \le x_2 \le\cdots$.
\end{longlist}
\end{theorem}

Thus, we add to $ B $ limits of all monotone sequences, iterate this
operation until stabilization and get $ \Cl(B) $, call it the closure
of $ B $. (It is not a noise-type Boolean algebra, unless $ B $ is
classical.)

%
%th1.7 #&#
\begin{theorem}\label{1d2}
Let $ B $ and $ \Cl(B) $ be as in Theorem \ref{1d1}, and
\[
C = \bigl\{ x \in\Cl(B)\dvtx\exists y \in\Cl(B) \> x \wedge y =
0_\La,
x \vee y = 1_\La\bigr\}
\]
[the set of all complemented elements of $\Cl(B)$]. Then
\begin{longlist}[(a)]
\item[(a)] $ C $ is a noise-type Boolean algebra such that $ B
\subset C
\subset\Cl(B) $;

\item[(b)] $ C $ contains every noise-type Boolean algebra $ C_1 $ satisfying
$ B \subset C_1 \subset\Cl(B) $.
\end{longlist}
\end{theorem}

%
%de1.8 #&#
\begin{definition}\label{1d3}
The noise-type Boolean algebra $ C $ of Theorem \ref{1d2} is called
the \emph{noise-type completion} of a noise-type Boolean algebra $ B
$.
\end{definition}

%
%ex1.9 #&#
\begin{example}\label{1d35}
Let $B$, $y_n$ and $\xi_n$ be as in Section~\ref{1b}. Then $ \Cl(B)
\setminus B $ consists of $\sigma$-fields of the form $ \sup_{n\in I}
y_n =
\sigma
( \{ \xi_n \xi_{n+1}\dvtx n \in I \} ) $ where $I$ runs over all
infinite subsets of $ \{ 1,2,\ldots\} $. The noise-type completion of
$B$ is $B$ itself.
\end{example}

If two noise-type Boolean algebras have the same closure, then clearly
they have the same completion.\vadjust{\goodbreak}

%
%pr1.10 #&#
\begin{proposition}\label{1d4}
If two noise-type Boolean algebras have the same closure, then they
have the same first chaos space.
\end{proposition}

Thus if $ \Cl(B_1) = \Cl(B_2) $, then classicality of $ B_1 $ is
equivalent to classicality of~$ B_2 $, and blackness of $ B_1 $ is
equivalent to blackness of $ B_2 $.

%
%qu1.11 #&#
\begin{question}
It follows from Theorem \ref{1d1} that the following conditions are
equivalent: $ \Cl(B) $ is a lattice; $ \Cl(B) $ is a complete lattice;
$ x \vee y \in\Cl(B) $ for all $ x,y \in\Cl(B) $. These conditions
are satisfied by every classical $ B $. Are they satisfied by some
nonclassical $ B $? By all nonclassical $ B $?
\end{question}

%s1.5 #&#
\subsection{On sufficient subalgebras}
\label{1e}

Let $ B, B_0 $ be noise-type Boolean algebras such that $ B_0 \subset
B $. Clearly, $ \Cl(B_0) \subset\Cl(B) $ and $ H^{(1)}(B_0) \supset
H^{(1)}(B) $. We say that:
\begin{itemize}
\item$ B_0 $ is dense in $ B $ if $ \Cl(B_0) = \Cl(B) $;
\item$ B_0 $ is sufficient in $ B $ if $ H^{(1)}(B_0) = H^{(1)}(B)
$.
\end{itemize}

If $ B_0 $ is sufficient in $ B $, then clearly, classicality of $ B_0
$ is equivalent to classicality of $ B $, and blackness of $ B_0 $ is
equivalent to blackness of $ B $.

A dense subalgebra is sufficient by Proposition \ref{1d4}.
Surprisingly, a
nondense subalgebra can be sufficient.

%
%de1.12 #&#
\begin{definition}\label{111}
A noise-type Boolean algebra $ B $ is \emph{atomless} if
\[
\inf_{x\in F} x = 0_\La
\]
for every ultrafilter $ F \subset B $.
\end{definition}

Recall that a set $ F \subset B $ is called a filter if for all $ x,y
\in B $
\begin{eqnarray*}
&x \in F,\qquad x \le y \quad\imply\quad y \in F,&
\\
&x,y \in F \quad\imply\quad x \wedge y \in F,&
\\
&0_\La\notin F;&
\end{eqnarray*}
a filter $ F $ is called ultrafilter if it is a maximal filter;
equivalently, if
\[
\forall x \in B\qquad \bigl( x \notin F \impl x' \in F \bigr).
\]

%
%th1.13 #&#
\begin{theorem}\label{1e2}
If a noise-type subalgebra is atomless, then it is sufficient.
\end{theorem}

Some applications of this result are mentioned in the end of
Section~\ref{1f}.

%s1.6 #&#
\subsection{On available examples and frameworks}
\label{1f}

Several examples of nonclassical noise-type Boolean algebras are
available in the literature but described in somewhat different
frameworks.\vadjust{\goodbreak}

According to Tsirelson and Vershik (\cite{TV}, Definition~1.2), a \emph
{measure factorization} over
a Boolean algebra $\A$ is a map $ \phi\dvtx \A\to\La$ such that $
\phi(a_1\wedge a_2) = \phi(a_1)\wedge\phi(a_2) $, $ \phi(a_1\vee
a_2) =
\phi(a_1)\vee\phi(a_2) $, $ \phi(0_\A)=0_\La$, $ \phi(1_\A
)=1_\La$,
and two $\sigma$-fields $\phi(a)$, $\phi(a')$ are independent (for
all $
a,a_1,a_2 \in\A$). In this case the image $ B = \phi(\A) \subset
\La
$ evidently is a noise-type Boolean algebra. A measure factorization
over $\A$ may be defined equivalently as a homomorphism $\phi$ from
$\A$ onto some noise-type Boolean algebra. Assuming that $\phi$ is an
isomorphism (which usually holds) we may apply several notions
introduced in \cite{TV} to noise-type Boolean algebras.

In particular, an element of the first chaos space $H^{(1)}(B)$ is the
same as a square integrable real-valued \emph{additive integral}
\cite{TV}, Definition~1.3 and Theorem~1.7. Complex-valued \emph{multiplicative
integrals} are also examined in \cite{TV}, Theorem~1.7; these generate a
$\sigma$-field that contains the $\sigma$-field generated by
$H^{(1)}(B)$. These two
$\sigma$-fields differ in the ``simplest nonclassical example'' of
Section~\ref{1b}. Namely, the latter $\sigma$-field consists of all
measurable
sets invariant under the sign change, while the former $\sigma$-field is the
whole $1_\La$, since the coordinates $ \xi_1,\xi_2,\ldots$ are
multiplicative integrals [indeed, $ \xi_1 = (\xi_1\xi_2) (\xi_2\xi_3)
\cdots(\xi_n\xi_{n+1}) \xi_{n+1} $]. A sufficient condition for
equality of the two $\sigma$-fields, given by \cite{TV}, Theorem~1.7,
is the
\emph{minimal up continuity condition} \cite{TV}, Definition~1.6: $
\sup_{x\in F} x' = 1_\La$ for every ultrafilter $ F \subset B $. This
is stronger than the condition $ \inf_{x\in F} x = 0_\La$ called
\emph{minimal down continuity} in \cite{TV}, Definition~1.6, and just
\emph{atomless} here (Definition \ref{111}). The ``continuous
example'' in
\cite{Ts04}, Section~1b, is atomless but violates the minimal up
continuity condition. The seemingly evident relation $ \sup_{x\in F}
x' = ( \inf_{x\in F} x )' $ may fail (see Section~\ref{1b}), since
$\sup$ and $\inf$ are taken in $\La$ rather than $B$; see also Remark
\ref{4a10}.

A wide class of countable atomless black noise-type Boolean algebras
is obtained in \cite{TV}, Section~4a, via combinatorial models on trees.

According to \cite{Ts04}, Definition~3c1, a \emph{continuous product of
probability spaces} (over $\R$) is a family $(x_{s,t})_{s<t}$ of
$\sigma$-fields $ x_{s,t} \in\La$ given for all $ s,t \in\R$, $ s<t $, such that
$ \sup_{s,t} x_{s,t} = 1_\La$ and
\[
x_{r,s} \otimes x_{s,t} = x_{r,t}\qquad \mbox{whenever }
r<s<t
\]
in the sense that $ x_{r,s} $ and $ x_{s,t} $ are independent and
generate $ x_{r,t} $. This is basically the same as a measure
factorization over the Boolean algebra $\A$ of all finite unions of
intervals $(s,t)$ treated modulo finite sets (see
\cite{Ts04}, Section~11a, for details).

According to Tsirelson \cite{Ts04}, Definition~3d1, a \emph{noise}
(over $\R$) is a
\emph{homogeneous} continuous product of probability spaces;
``homogeneous'' means existence of a measurable action $(T_h)_{h\in\R}
$ of $\R$ on $\Om$ such that
\[
T_h \mbox{ sends } x_{s,t} \mbox{ to } x_{s+h,t+h}\qquad
\mbox{whenever } s<t \mbox{ and } h\in\R
\]
(see \cite{Ts04}, Section~3d for details). It follows from homogeneity
(and separability of~$H$) that \cite{Ts04}, Proposition~3d3 and Corollary~3d5
%
%
%e1.1 #&#
%
\begin{equation}
\label{*} \inf_{\eps>0} x_{s-\eps,t+\eps} = x_{s,t} = \sup
_{\eps>0} x_{s+\eps,t-\eps},
\end{equation}
which implies the minimal up continuity condition (since an
ultrafilter must contain all neighborhoods of some point from
$[-\infty,+\infty]$). Thus, additive and multiplicative integrals
generate the same sub-$\sigma$-field, called the \emph{stable $\sigma
$-field} in
\cite{Ts04}, Section~4c, where it is defined in a completely different
but equivalent way. Note also that every noise leads to an
\emph{atomless} noise-type Boolean algebra.

Two examples of a nonclassical, but not black, noise were published in
1999 and 2002 by J.~Warren (see \cite{Ts04}, Sections~2c, 2d).

Existence of a black noise was proved first in 1998 (\cite{TV},
Section~5),
via projective limit; see also \cite{Ts03}, Section~8.2. However, this
was not quite a \emph{construction} of a specific noise; existence of
a subsequence limit was proved, uniqueness was not.

All other black noise examples available for now use random
configurations over $\R^{1+d}$ for some $d\ge1$ (in most cases $d=1$);
the $\sigma$-field $x_{s,t}$ consists of all events ``observable'' within
the domain $ (s,t)\times\R^d \subset\R^{1+d} $.

Examples based on stochastic flows were published in 2001 by
Watanabe and in 2004 by the author Le Jan, O.~Raimond and
S.~Lemaire. In these examples the first coordinate of $ \R^{1+d} $ is
interpreted as time, the other $d$ coordinates as space. Blackness is
deduced from the relation $ \| \mathbb{E}(f |x_{t,t+\eps} ) \|^2 =
o(\eps)
$ as $ \eps\to0+ $ for all $ f \in L_2(\Om,\F,P) $ such that $ \Ex
f =
0 $. For details and references see \cite{Ts04}, Section~7.

The first highly important example is the \emph{black noise of
percolation.} The corresponding random configuration over $ \R^2 $ is
the full scaling limit of critical site percolation on the triangular
lattice. This example was conjectured in 2004 (\cite{Ts03}, Question
8.1 and
Remark 8.2, \cite{Ts04}, Question 11b1). It was rather clear
that the noise of percolation must be black; it was less clear how to
define its probability space and $\sigma$-fields $ x_{s,t} $, and it was
utterly unclear whether $ x_{r,s} $ and $ x_{s,t} $ generate $ x_{r,t}
$, or not. (It is not sufficient to know that $ x_{r,s+\eps} $ and $
x_{s,t} $ generate $ x_{r,t} $.) The affirmative answer was published
in 2011~\cite{SS}.\looseness=-1

In order to say that the noise of percolation is a conformally
invariant black noise over $\R^2$ we must first define a noise over
$\R^2$. Recall that a noise over $\R$ is related to the Boolean
algebra of all finite unions of intervals modulo finite sets. Its
two-dimensional counterpart, according to Schramm and Smirnov \cite
{SS}, Corollary 1.20,
is ``an appropriate algebra of piecewise-smooth planar domains
(e.g., generated by rectangles).'' However, the algebra generated by
rectangles hides the conformal invariance of this noise. The class of
all piecewise-smooth domains is conformally invariant, however, two
$C^k$-smooth curves may have a nondiscrete intersection. Piecewise
analytic boundaries could be appropriate for this noise.\vadjust{\goodbreak}

Stochastic flows on $ \R^{1+d} $, mentioned above, lead to noises over
$\R$, generally not $ \R^{1+d} $ since, being uncorrelated in time,
they may be correlated in space. However, two of them are also
uncorrelated in (one-dimensional) space: Arratia's coalescing flow, or
the Brownian web (see~\cite{Ts04}, Section~7f), and its sticky
counterpart (see~\cite{Ts04}, Section~7j). For such flow it is natural to
conjecture that a $\sigma$-field $ y_{a,b} $ consisting of all events
``observable'' within the domain $ \R\times(a,b) \subset\R^2 $ is
well defined whenever $a<b$, and $ y_{a,b} \otimes y_{b,c} = y_{a,c}
$. Then $(y_{a,b})_{a<b}$ is the second noise (over $\R$) obtained
from this flow. Moreover, the $\sigma$-fields $ x_{s,t} \wedge y_{a,b} $
indexed by rectangles $ (s,t)\times(a,b) $ should form a noise over
$\R^2$. For Arratia's flow this conjecture was proved in 2011
\cite{EF}. It appears that the relation $ y_{a,b} \otimes y_{b,c} =
y_{a,c} $ is harder to prove than the relation $ x_{r,s} \otimes
x_{s,t} = x_{r,t} $. Unlike percolation, Arratia's flow, being
translation-invariant (in time and space), is not
rotation-invariant, and the two noises $(x_{s,t})_{s<t}$,
$(y_{a,b})_{a<b}$ are probably nonisomorphic.

Still, the notion of a noise over $\R^2$ is obscure because of
nonuniqueness of an appropriate Boolean algebra of planar
domains. Surely, a single ``noise of percolation'' is more
satisfactory than ``the noise of percolation on rectangles'' different
from ``the noise of percolation on piecewise analytic domains''
etc. These should be treated as different generators of the same
object. On the level of noise-type Boolean algebras the problem is
solved by the noise-type completion (Section~\ref{1d}). However, it
remains unclear how to relate the $\sigma$-fields belonging to the completion
to something like planar domains.

Any reasonable definition of a noise over $\R^2$ leads to a noise-type
Boolean algebra $B$, two noises $(x_{s,t})_{s<t}$, $(y_{a,b})_{a<b}$
over $\R$, their noise-type Boolean algebras $ B_1 \subset B $, $ B_2
\subset B $, and the corresponding first chaos spaces $ H^{(1)}(B) $,
$ H^{(1)}(B_1) $, $ H^{(1)}(B_2) $. As was noted after \eqref{*}, $
B_1 $ and $ B_2 $ are atomless. By Theorem \ref{1e2} they are
sufficient, that is,
\[
H^{(1)}(B_1) = H^{(1)}(B)= H^{(1)}(B_2)
.
\]
Thus, if one of these three noises (one over $\R^2$ and two over $\R$)
is classical, then the other two are classical; if one is black, then
the other two are black.

For the noise of percolation we know that the noise over $\R^2$ is
black and conclude that the corresponding two (evidently isomorphic)
noises over $\R$ are black.

For the Arratia's flow we know that the first noise over $\R$ is black
and conclude that the second noise over $\R$ is also black.

%s2 #&#
\section{Preliminaries}
\label{sec2}

This section is a collection of useful facts (mostly folk-lore, I
guess), more general than noise-type Boolean algebras.

Throughout, the probability space $ (\Om,\F,P) $, the complete lattice
$ \La$ of sub-$\sigma$-fields and the separable Hilbert space $ H =
L_2(\Om,\F,P) $ are as in Section~\ref{1a}. Complex numbers are not
used; $ H $ is a Hilbert space over $ \R$. A ``subspace'' of $ H $
always means a closed linear subset. Recall also $ 0_\La, 1_\La,
x\wedge y, x\vee y $ for $ x,y \in\La$, the notion of
independent $\sigma$-fields, operators $ \mathbb{E}(\cdot |x)$ of conditional
expectation, and $ \inf X, \sup X \in\La$ for $ X \subset\La$
(Section~\ref{1a}).

%s2.1 #&#
\subsection{Type $ L_2 $ subspaces}
\label{2a}

%
%fa2.1 #&#
\begin{fact}[(\cite{Si}, \textup{Theorem}~3)]\label{2a1}
The following two conditions on a subspace $ H_1 $ of $ H $ are
equivalent:
\begin{longlist}[(a)]
\item[(a)] there exists a sub-$\sigma$-field $ x \in\La$ such
that $ H_1 =
L_2(x) $, the space of all \mbox{$x$-measurable} functions of $ H $;

\item[(b)] $ H_1 $ is a sublattice of $ H $, containing constants.
That is, $
H_1 $ contains $ f \vee g $ and $ f \wedge g $ for all $ f,g \in H_1
$, where $ (f \vee g)(\omega) = \max(f(\omega),g(\omega)) $ and $
(f \wedge
g)(\omega) = \min(f(\omega),g(\omega)) $, and $ H_1 $ contains the
one-dimensional space of constant functions.
\end{longlist}
\end{fact}

% \begin{pf*}[Hint \textrm{to the proof that (b) \imp(a)}]
%
\begin{pf*}{Hint to the proof that (b) \imp(a)}
$ \One_{(0,\infty)}(f) = \lim_n ( ( 0 \vee nf ) \wedge1 ) \in H_1 $
for $ f \in H_1 $.
\end{pf*}

Such subspaces $ H_1 $ will be called type $ L_2 $ (sub)spaces. (In
\cite{Si} they are called measurable, which can be confusing.)

Due to linearity of $ H_1 $ the condition $ f \vee g, f \wedge g \in
H_1 $ boils down to $ |f| \in H_1 $ for all $ f \in H_1 $. [Hint: $
f \vee g = f + ( 0 \vee(g-f) ) $ and $ 0 \vee f = 0.5 ( f + |f| )
$.]

%
%fa2.2 #&#
\begin{fact}\label{2a2}
If $ A \subset L_\infty(\Om,\F,P) $ is a subalgebra containing
constants, then the closure of $ A $ in $ H $ is a type $ L_2 $ space.
\end{fact}

(``Subalgebra'' means $ fg \in A $ for all $ f,g \in A $, in addition
to linearity.)

\begin{pf*}{Hint}
Approximating the absolute value by polynomials we get $ |f| \in
H_1 $ (the closure of $ A $) for $ f \in A $, and by continuity, for $
f \in H_1 $.
\end{pf*}

%
%no2.3 #&#
\begin{notation}\label{2a3}
We denote the type $ L_2 $ space $ L_2(x) $ corresponding to $
x \in\La$ by $ H_x $, and the orthogonal projection $ \mathbb{E}(\cdot|x)$
by $ Q_x $. In particular, $ H_0 = \{ c \One\dvtx c \in\R\} $ is the
one-dimensional subspace of constant functions on $ \Om$, and $ Q_0 f
= ( \Ex f ) \One= \langle f,\One\rangle\One$. Also, $ H_1 = H $,
and $ Q_1 =
I $ is the identity operator.
\end{notation}

Thus:
%
%
%e2.1 #&#
%e2.2 #&#
%e2.3 #&#
%e2.4 #&#
%e2.5 #&#
%e2.6 #&#
%
\begin{eqnarray}
&H_x \subset H;\qquad Q_x\dvtx H \to H;\qquad Q_x H =
H_x \qquad\mbox{for } x \in\La; &\label{2a4}
\\
&H_x \subset H_y \quad\equiv \quad Q_x \le
Q_y \quad\equiv\quad x \le y;& \label{2a5}
\\
&Q_x Q_y = Q_x = Q_y
Q_x\qquad \mbox{whenever } x \le y;& \label{2a6}
\\
&H_x = H_y \quad\equiv\quad Q_x = Q_y
\quad\equiv\quad x = y;& \label{2a7}
\\
&H_{x\wedge y} = H_x \cap H_y; &\label{2a8}
\end{eqnarray}
\eqref{2a6} and \eqref{2a7} follow from \eqref{2a5}; \eqref{2a8} is a
special case of Fact \ref{2a9}.\vadjust{\goodbreak}

%
%fa2.4 #&#
\begin{fact}\label{2a9}
$ H_{\inf X} = \bigcap_{x\in X} H_x $ for $ X \subset\La$.
\end{fact}

\begin{pf*}{Hint}
Measurability w.r.t. the intersection of $\sigma$-fields is equivalent
to measurability w.r.t. each one of these $\sigma$-fields.
\end{pf*}

However, $ H_{x\vee y} $ is generally much larger than the closure of
$ H_x + H_y $.

%
%fa2.5 #&#
\begin{fact}[(\cite{Ma}, \textup{Theorem}~3.5.1)]\label{2a10}
$ H_{x\vee y} $ is the subspace spanned by pointwise products $ fg $
for $ f \in H_x \cap L_\infty(\Om,\F,P) $ and $ g \in H_y \cap
L_\infty(\Om,\F,P) $.
\end{fact}

\begin{pf*}{Hint}
Linear combinations of these products are an algebra;
by Fact \ref{2a2} its closure is $ H_z $ for some $ z \in\La$; note that
$ z \ge x $, $ z \ge y $, but also $ z \le x \vee y $.
\end{pf*}

%
%fa2.6 #&#
\begin{fact}\label{2a11}
Let $ x,x_1,x_2,\ldots\in\La$, $ x_1 \le x_2 \le\cdots$ and $ x
= \sup_n x_n $. Then $ H_x $ is the closure of $ H_{x_1} \cup
H_{x_2} \cup\cdots.$
\end{fact}

\begin{pf*}{Hint}
By Fact \ref{2a1}, the closure of $ \bigcup_n H_{x_n}$ is $ H_z $
for some $ z \in\La$; note that $ z \ge x_n $ for all $ n $, but
also $ z \le x $.
\end{pf*}

That is,
%
%
%e2.7 #&#
%e2.8 #&#
%
\begin{eqnarray}
x_n \uparrow x \quad&\impl&\quad H_{x_n} \uparrow H_x;
\label{2a12}
\\
x_n \downarrow x \quad&\imply&\quad H_{x_n} \downarrow H_x
\label{2a13}
\end{eqnarray}
(the latter holds by Fact \ref{2a9}).

%s2.2 #&#
\subsection{Strong operator convergence}
\label{2b}

Let $ H $ be a Hilbert space and $ A, A_1,\break   A_2, \ldots\dvtx H \to H $
operators (linear, bounded). Strong operator convergence of $ A_n $ to
$ A $ is defined by
\[
( A_n \to A ) \quad\equiv\quad\bigl( \forall\psi\in H \> \| A_n
\psi- A \psi\| \mathop{\longrightarrow}^{n\to\infty}0 \bigr).
\]
We write just $ A_n \to A $, since we do not need other types of
convergence for operators.

%
%fa2.7 #&#
\begin{fact}[(\cite{Pe}, Remark~2.2.11)]\label{2b05}
$ A_n \to A $ if and only if $ \| A_n \psi- A \psi\| \to0 $ for a
dense set of vectors $ \psi$ and $ \sup_n \| A_n \| < \infty$.
\end{fact}

%
%fa2.8 #&#
\begin{fact}[(\cite{Ha1}, \textup{Problem} 93; \cite{Pe}, Section~4.6.1)]\label{2b1}
If $ A_n \to A $ and $ B_n \to B $, then $ A_n B_n \to A B $.
\end{fact}

%
%fa2.9 #&#
\begin{fact}\label{2b2}
If $ A_n \to A $, $ B_n \to B $ and $ A_n B_n = B_n A_n $ for all $ n
$, then $ AB = BA $.
\end{fact}

\begin{pf*}{Hint}
Use Fact \ref{2b1}.\vadjust{\goodbreak}
\end{pf*}

The following fact allows us to write $ A_n \uparrow A $ (or $
A_n \downarrow A $) unambiguously. We need it only for commuting
orthogonal projections.

%
%fa2.10 #&#
\begin{fact}[(\cite{Co}, Proposition~43.1)]\label{2b25}
Let $ A, A_1, A_2, \ldots\dvtx H \to H $ be Hermitian operators, $ A_1
\le
A_2 \le\cdots$, then
\[
A = \sup_n A_n \quad\equiv\quad A_n \to A.
\]
\end{fact}

The natural bijective correspondence between subspaces of $ H $ and
orthogonal projections $ H \to H $ is order preserving, therefore
%
%
%e2.9 #&#
%
\begin{equation}
\label{2b3}\qquad H_n \downarrow H_\infty\quad\equiv\quad Q_n
\downarrow Q_\infty\quad \mbox{also}\quad H_n \uparrow
H_\infty\quad\equiv\quad Q_n \uparrow Q_\infty
\end{equation}
whenever $ H_1, H_2, \ldots, H_\infty\subset H $ are subspaces and $
Q_1, Q_2, \ldots, Q_\infty\dvtx H \to H $ the corresponding orthogonal
projections.

In combination with \eqref{2a12}, \eqref{2a13} it gives
%
%
%e2.10 #&#
%
\begin{equation}
\label{2b4} x_n \downarrow x \mbox{ implies } Q_{x_n}
\downarrow Q_x; \mbox{also, } x_n \uparrow x \mbox{
implies } Q_{x_n} \uparrow Q_x.
\end{equation}

Let $ H_1, H_2 $ be Hilbert spaces, and $ H = H_1 \otimes H_2 $ their
tensor product.

%
%fa2.11 #&#
\begin{fact}\label{2b5}
Let $ A,A_1,A_2,\ldots\dvtx H_1 \to H_1 $, $ B,B_1,B_2,\ldots\dvtx
H_2 \to H_2
$. If\break  $ A_n \to A $ and $ B_n \to B $, then $ A_n \otimes B_n \to
A \otimes B $.
\end{fact}

\begin{pf*}{Hint}
The operators are uniformly bounded, and converge on a dense set;
use Fact \ref{2b05}.
\end{pf*}

%s2.3 #&#
\subsection{Independence and tensor products}
\label{2c}

%
%fa2.12 #&#
\begin{fact}\label{2c1}
If $ x,y \in\La$ are independent, then $ H_{x\vee y} = H_x \otimes
H_y $ up to the natural unitary equivalence:
\[
H_x \otimes H_y \ni f \otimes g \quad\longleftrightarrow \quad fg
\in H_{x\vee y}.
\]
\end{fact}

\begin{pf*}{Hint}
By the independence, $ \langle f_1 g_1,f_2 g_2 \rangle= \Ex( f_1 g_1
f_2 g_2 ) = \Ex( f_1 f_2 )\times\break   \Ex( g_1 g_2 ) = \langle f_1,f_2 \rangle
\langle g_1,g_2 \rangle= \langle f_1 \otimes g_1,f_2 \otimes g_2
\rangle$, thus,
$ H_x \otimes H_y $ is isometrically embedded into $ H_{x\vee y} $;
by Fact \ref{2a10} the embedding is ``onto.''
\end{pf*}

It may be puzzling that $ H_x $ is both a subspace of $ H $ and a
tensor factor of $ H $ (which never happens in the general theory of
Hilbert spaces). Here is an explanation. All spaces $ H_x $ contain
the one-dimensional space $ H_0 $ of constant functions (on $ \Om
$). Multiplying an $x$-measurable function $ f \in H_x $ by the
constant function $ g \in H_{x'} $, $ g(\cdot)=1 $, we get the
(puzzling) equality $ f \otimes g = f $.

%
%no2.13 #&#
\begin{notation}\label{2c2}
For $ u,x \in\La$ such that $ u \le x $ we denote by $ Q_u^{(x)} $
the restriction of $ Q_u $ to $ H_x $.\vadjust{\goodbreak}
\end{notation}

Thus
\[
Q_u^{(x)}\dvtx H_x \to H_x,\qquad
Q_u^{(x)} H_x = H_u \qquad\mbox{for } u
\le x.
\]

%
%fa2.14 #&#
\begin{fact}\label{2c3}
If $ x,y \in\La$ are independent, $ u \le x $, $ v \le y $, then
treating $ H_{x\vee y} $ as $ H_x \otimes H_y $, we have
\[
Q_{u\vee v} = Q_u^{(x)} \otimes Q_v^{(y)}
.
\]
\end{fact}

\begin{pf*}{Hint}
By Fact \ref{2c1}, $ H_{u\vee v} = H_u \otimes H_v $, and this
factorization may be treated as embedded into the factorization $
H_{x\vee y} = H_x \otimes H_y $; the projection onto $ H_u \otimes
H_v \subset H_x \otimes H_y $ factorizes.
\end{pf*}

In a more probabilistic language,
\[
\mathbb{E}(fg |u\vee v ) = \mathbb{E}(f |u) \mathbb{E}(g |v)\qquad \mbox
{for } f \in
L_2(x), g \in L_2(y).
\]

Here is a very general fact (no $\sigma$-fields, no tensor products, just
Hilbert spaces).

%
%fa2.15 #&#
\begin{fact}[(\cite{Ha1}, \textup{Problem} 96, \cite{Co}, \textup{Exercise} 45.4)]\label{2c4}
Let $ Q_1, Q_2 $ be orthogonal projections in a Hilbert space $ H
$. Then $ (Q_1 Q_2)^n $ converges strongly (as~$ n \to\infty$) to
the orthogonal projection onto $ (Q_1 H) \cap(Q_2 H) $.
\end{fact}

%
%fa2.16 #&#
\begin{fact}\label{2c5}
$ (Q_x Q_y)^n \to Q_{x\wedge y} $ strongly (as $ n \to\infty$)
whenever $ x,y \in\La$.
\end{fact}

\begin{pf*}{Hint}
$ (Q_x H) \cap(Q_y H) = Q_{x\wedge y} H $ by \eqref{2a8};
use Fact \ref{2c4}.
\end{pf*}
%

%
%fa2.17 #&#
\begin{fact}\label{2c6}
$ (Q_{u_1}^{(x)} Q_{u_2}^{(x)})^n \to Q_{u_1 \wedge u_2}^{(x)} $
strongly (as $ n \to\infty$) whenever\break  \mbox{$ u_1, u_2 \le x $}.
\end{fact}

\begin{pf*}{Hint}
Similar to Fact \ref{2c5}.
\end{pf*}

%
%fa2.18 #&#
\begin{fact}\label{2c7}
If $ x,y \in\La$ are independent, $ u_1, u_2 \le x $ and $ v_1,
v_2 \le y $, then
\[
( u_1 \vee v_1 ) \wedge( u_2 \vee
v_2 ) = ( u_1 \wedge u_2 ) \vee(
v_1 \wedge v_2 ).
\]
\end{fact}

\begin{pf*}{Hint}
By Fact \ref{2c5}, $ ( Q_{u_1 \vee v_1} Q_{u_2 \vee v_2} )^n \to Q_{(
u_1 \vee v_1 ) \wedge( u_2 \vee v_2 )} $. By Fact \ref{2c6}, $
(Q_{u_1}^{(x)} Q_{u_2}^{(x)})^n \to Q_{u_1 \wedge u_2}^{(x)} $ and $
(Q_{v_1}^{(y)} Q_{v_2}^{(y)})^n \to Q_{v_1 \wedge v_2}^{(y)}
$. By Fact \ref{2c3}, $ Q_{u_1 \vee v_1}\times  Q_{u_2 \vee v_2} = (
Q_{u_1}^{(x)} \otimes Q_{v_1}^{(y)} )( Q_{u_2}^{(x)} \otimes
Q_{v_2}^{(y)} ) = ( Q_{u_1}^{(x)} Q_{u_2}^{(x)} ) \otimes(
Q_{v_1}^{(y)} Q_{v_2}^{(y)} ) $ and\break  $ Q_{ ( u_1 \wedge u_2 ) \vee(
v_1 \wedge v_2 ) } = Q_{u_1\wedge u_2}^{(x)} \otimes Q_{v_1\wedge
v_2}^{(y)} $; use \eqref{2a7}.
\end{pf*}

\begin{remark*}
In a distributive lattice the equality stated
by Fact \ref{2c7} is easy to check (assuming $ x\wedge y = 0 $ instead of
independence). However, the lattice $ \La$ is not distributive.
\end{remark*}

Useful special cases of Fact \ref{2c7} (assuming that $ x,y $ are
independent, $ u \le x $ and $ v \le y $):
%
%
%e2.11 #&#
%e2.12 #&#
%
\begin{eqnarray}
&( u \vee v ) \wedge x = u,\qquad ( u \vee v ) \wedge y = v;& \label{2c8}
\\
&( u \vee y ) \wedge( x \vee v ) = u \vee v.& \label{2c9}
\end{eqnarray}

Here is another very general fact (no $\sigma$-fields, no tensor
products, just
random variables).

%
%fa2.19 #&#
\begin{fact}\label{2c10}
Assume that $ X,X_1,X_2,\ldots$ and $ Y,Y_1,Y_2,\ldots$ are random variables
(on a given probability space), and for every $ n $ the two random
variables $ X_n, Y_n $ are independent; if $ X_n \to X $, $ Y_n \to Y
$ in probability, then $ X,Y $ are independent.
\end{fact}

\begin{pf*}{Hint}
If $ f,g\dvtx\R\to\R$ are bounded continuous functions, then $
\Ex(f(X_n)) \to\Ex(f(X)) $, $ \Ex(g(Y_n)) \to\Ex(g(Y)) $, $
\Ex(f(X_n)) \Ex(g(Y_n)) = \Ex(f(X_n)g(Y_n)) \to\break  \Ex(f(X)g(Y)) $, thus,
$ \Ex(f(X)g(Y)) = \Ex(f(X)) \Ex(g(Y)) $.
\end{pf*}

The same holds for vector-valued random variables.

%s2.4 #&#
\subsection{Measure class spaces and commutative von Neumann algebras}
\label{2d}

See \cite{Di,Ta} or \cite{Co} for basics about von Neumann
algebras; we need only the commutative case.

%
%fa2.20 #&#
\begin{fact}[(\cite{Di}, \textup{Section~I.7.3}, \cite{Ta}, \textup{Theorem~III.1.22},
\cite{Pe}, \textup{E4.7.2})]\label{2d1}
Every commutative von Neumann algebra $ \A$ of operators on a
separable Hilbert space~$ H $ is isomorphic to the algebra $ L_\infty
(S,\Si,\mu) $ on some measure space $ (S,\Si,\mu) $.
\end{fact}

Here and henceforth all measures are positive, finite and such that
the corresponding $ L_2 $ spaces are separable. The isomorphism $ \al
\dvtx\A\to L_\infty(S,\Si,\mu) $ preserves linear operations,
multiplication and norm. Hermitian operators of $ \A$ correspond to
real-valued functions of $ L_\infty$; we restrict ourselves to these
and observe an order isomorphism,
%
%
%e2.13 #&#
%
\begin{eqnarray}
\label{2d11} A& \le& B\quad \equiv\quad\al(A) \le\al(B);
\nonumber
\\[-8pt]
\\[-8pt]
\nonumber
 A &=& \sup_n
A_n \quad\equiv\quad \al(A) = \sup_n \al(A_n).
\end{eqnarray}

%
%fa2.21 #&#
\begin{fact}[(\cite{Di}, \textup{Section}~I.4.3, \textup{Corollary} 1, \cite{Co}, \textup{Section}~46,
\textup{Proposition}~46.6 \textup{and
Exercise} 1)]\label{2d17}
Every isomorphism of von Neumann algebras preserves the strong
operator convergence (of sequences, not nets).
\end{fact}

The measure $ \mu$ may be replaced with any equivalent (i.e.,
mutually absolutely continuous) measure $ \mu_1 $. Thus we may turn to
a measure class space (see \cite{Ar}, Section~14.4) $ (S,\Si,\M) $ where
$ \M$ is an equivalence class of\vadjust{\goodbreak} measures, and write $
L_\infty(S,\Si,\M) $; we have an isomorphism
%
%
%e2.14 #&#
%
\begin{equation}
\label{2d13} \al\dvtx \A\to L_\infty(S,\Si,\M)
\end{equation}
of von Neumann algebras. (See \cite{Ar}, Section~14.4, for the Hilbert space $
L_2(S,\break \Si,\M) $ on which $ L_\infty(S,\Si,\M) $ acts by
multiplication.)

%
%fa2.22 #&#
\begin{fact}\label{2d2}
Let $ \A$ and $ \al$ be as in \eqref{2d13}, $ A,A_1,A_2,\ldots\in\A
$, $ \sup_n \|A_n\| < \infty$. Then the following two conditions are
equivalent:
\begin{longlist}[(a)]
\item[(a)] $ A_n \to A $ in the strong operator topology;

\item[(b)] $ \al(A_n) \to\al(A) $ in measure.
\end{longlist}
\end{fact}

\begin{pf*}{Hint}
($ A_n \to A $ strongly)\,\equ\,($ \al(A_n) \to\al(A) $
strongly)\,\equ\,($ \| \al(A_n)f - \al(A)f \|_2 \to0 $ for every bounded
$ f $)\,\equ\,($ \al(A_n) \to\al(A) $ in measure).
\end{pf*}

Let $ \Si_1 \subset\Si$ be a sub-$\sigma$-field. Restrictions $ \mu
|_{\Si
_1} $
of measures $ \mu\in\M$ are mutually equivalent; denoting their
equivalence class by $ \M|_{\Si_1} $ we get a measure class space $
(S,\Si_1,\M|_{\Si_1}) $. Clearly, $ L_\infty(S,\Si_1,\M|_{\Si_1})
\subset L_\infty(S,\Si,\M) $ or, in shorter notation, $
L_\infty(\Si_1) \subset L_\infty(\Si) $; this is also a von Neumann
algebra.

%
%fa2.23 #&#
\begin{fact}\label{2d3}
Every von Neumann subalgebra of $ L_\infty(\Si) $ is $ L_\infty(\Si_1)
$ for some sub-$\sigma$-field $ \Si_1 \subset\Si$.
\end{fact}

\begin{pf*}{Hint}
Similar to Fact \ref{2a2}.
\end{pf*}

We have $ L_\infty(\Si_1) = \al(\A_1) $ where $ \A_1 = \al^{-1} (
L_\infty(\Si_1) ) \subset\A$ is a von Neumann algebra. And
conversely, if $ \A_1 \subset\A$ is a von Neumann algebra, then
$ \al(\A_1) = L_\infty(\Si_1) $ for some sub-$\sigma$-field $ \Si_1
\subset\Si
$.

Given two von Neumann algebras $ \A_1, \A_2 \subset\A$, we denote by
$ \A_1 \vee\A_2 $ the von Neumann algebra generated by $ \A_1, \A_2
$. Similarly, for two $\sigma$-fields $ \Si_1, \Si_2 \subset\Si$
we denote by
$ \Si_1 \vee\Si_2 $ the $\sigma$-field generated by $ \Si_1, \Si
_2 $.

%
%fa2.24 #&#
\begin{fact}\label{2d4}
$ L_\infty(\Si_1) \vee L_\infty(\Si_2) = L_\infty(\Si_1 \vee\Si
_2) $.
\end{fact}

\begin{pf*}{Hint}
By Fact \ref{2d3}, $ L_\infty(\Si_1) \vee L_\infty(\Si_2) =
L_\infty(\Si_3)
$ for some $ \Si_3 $; note that $ \Si_3 \supset\Si_1 $,
$ \Si_3 \supset\Si_2 $, but also $ \Si_3 \subset\Si_1 \vee\Si_2 $.
\end{pf*}

%
%fa2.25 #&#
\begin{fact}\label{2d5}
If $ \al(\A_1) = L_\infty(\Si_1) $ and $ \al(\A_2) = L_\infty
(\Si_2), $
then $ \al(\A_1 \vee\A_2) = L_\infty(\Si_1 \vee\Si_2) $.
\end{fact}

\begin{pf*}{Hint}
$ \al(\A_1 \vee\A_2) = \al(\A_1) \vee\al(\A_2) $, since $ \al
$ is
an isomorphism; use Fact~\ref{2d4}.
\end{pf*}

The product $ (S,\Si,\M) = (S_1,\Si_1,\M_1) \times(S_2,\Si_2,\M
_2) $
of two measure class spaces is a measure class space \cite{Ar}, 14.4;
namely, $ (S,\Si) = (S_1,\Si_1) \times(S_2,\Si_2) $, and $ \M$ is\vadjust{\goodbreak}
the equivalence class containing $ \mu_1 \times\mu_2 $ for some
(therefore all) $ \mu_1 \in\M_1 $, $ \mu_2 \in\M_2 $. In this
case $
L_\infty(S,\Si,\M) = L_\infty(S_1,\Si_1,\M_1) \otimes
L_\infty(S_2,\break \Si_2,\M_2) $.

Given two commutative von Neumann algebras $ \A_1 $ on $ H_1 $ and $
\A_2 $ on $ H_2 $, their tensor product $ \A= \A_1 \otimes\A_2 $ is
a von Neumann algebra on $ H = H_1 \otimes H_2 $. Given isomorphisms $
\al_1\dvtx\A_1 \to L_\infty(S_1,\Si_1,\M_1) $ and $ \al_2\dvtx
\A_2
\to
L_\infty(S_2,\Si_2,\M_2) $, we get an isomorphism $ \al= \al_1
\otimes\al_2\dvtx\A\to L_\infty(S,\Si,\M) $, where $ (S,\Si,\M
) =
(S_1,\Si_1,\M_1) \times(S_2,\Si_2,\M_2) $; namely, $ \al( A_1
\otimes A_2 ) = \al_1 (A_1) \otimes\al_2 (A_2) $ for $ A_1 \in\A_1
$, $ A_2 \in\A_2 $. Note that $ \al( \A_1 \otimes I ) = L_\infty
(\ti\Si_1) $ and $ \al( I \otimes\A_2 ) = L_\infty(\ti\Si_2) $,
where $ \ti\Si_1 = \{ A_1 \times S_2\dvtx A_1 \in\Si_1 \} $ and $
\ti\Si_2 = \{ S_1 \times A_2\dvtx A_2 \in\Si_2 \} $ are $\M$-independent
sub-$\sigma$-fields of $ \Si$, and $ \ti\Si_1 \vee\ti\Si_2 = \Si$.

%
%de2.26 #&#
\begin{definition}
Let $ (S,\Si,\M) $ be a measure class space. Two sub-$\sigma$-fields
$ \Si_1,
\Si_2 \subset\Si$ are \emph{$\M$-independent}, if they are
$\mu$-independent for some $ \mu\in\M$, that is, $ \mu(X \cap
Y) \mu(S) = \mu(X) \mu(Y) $ for all $ X \in\Si_1 $, $ Y \in\Si_2 $.
\end{definition}

%
%fa2.27 #&#
\begin{fact}\label{2d7}
If $\sigma$-fields $ \Si_1, \Si_2 \subset\Si$ are independent,
then $
L_\infty(\Si_1 \vee\Si_2) = L_\infty(\Si_1) \otimes L_\infty(\Si
_2) $
up to the natural isomorphism
\[
L_\infty(\Si_1) \otimes L_\infty(\Si_2)
\ni f \otimes g \quad\longleftrightarrow\quad fg \in L_\infty(\Si_1
\vee\Si_2).
\]
\end{fact}

\begin{pf*}{Hint}
Recall Fact \ref{2c1}.
\end{pf*}

%
%fa2.28 #&#
\begin{fact}\label{2d8}
For every isomorphism $ \al\dvtx \A_1 \otimes\A_2 \to L_2(S,\Si
,\M), $
there exist $\M$-independent $ \Si_1,\Si_2 \subset\Si$ such that $
\al(\A_1 \otimes I) = L_\infty(\Si_1) $, $ \al(I \otimes\A_2)
= L_\infty(\Si_2) $, and $ \Si_1 \vee\Si_2 = \Si$.
\end{fact}

\begin{pf*}{Hint}
We get $ \Si_1, \Si_2 $ from Fact \ref{2d3}; $ \Si_1 \vee\Si_2 =
\Si$
by Fact \ref{2d4}; for proving independence we choose $ \mu_1 \in\M
_1 $,
$ \mu_2 \in\M_2 $, take isomorphisms $ \al_1\dvtx\A_1 \to
L_\infty(S_1,\Si'_1,\M_1) $, $ \al_2\dvtx\A_2 \to
L_\infty(S_2,\Si'_2,\M_2) $ and use the isomorphism $ \be=
(\al_1 \otimes\al_2) \al^{-1}\dvtx L_\infty(S,\Si,\M) \to
L_\infty(
(S_1,\Si'_1,\M_1) \times(S_2,\Si'_2,\M_2) ) $ for defining
$ \mu\in\M$ by $ \int f \,\D\mu= \int(\be
f) \,\D(\mu_1 \times\mu_2) $; then $ \Si_1,\Si_2 $
are $\mu$-independent.
\end{pf*}

Given an isomorphism $ \al\dvtx \A\to L_\infty(S,\Si,\M) $ of von
Neumann algebras, we have subspaces $ H(E) $, for $ E \in\Si$, of
the space $ H $ on which acts $ \A$:
%
%
%e2.15 #&#
%
\begin{eqnarray}
\label{2d9} H (E) &=& \al^{-1}(\One_E) H \subset H;
\nonumber\\
H ( E_1 \cap E_2 ) &=& H(E_1) \cap
H(E_2);
\nonumber
\\[-8pt]
\\[-8pt]
\nonumber
H ( E_1 \uplus E_2 ) &= &H(E_1) \oplus
H(E_2);
\\
H ( E_1 \cup E_2 )& =& H(E_1) +
H(E_2)\nonumber
\end{eqnarray}
[the third line differs from the fourth line by assuming that $ E_1,
E_2 $ are disjoint and concluding that $ H(E_1), H(E_2) $ are
orthogonal]. By \eqref{2d11}, \eqref{2b3}
%
%
%e2.16 #&#
%
\begin{eqnarray}
\label{2d10}& E_n \uparrow E \mbox{ implies } H(E_n)
\uparrow H(E),&
\nonumber
\\[-8pt]
\\[-8pt]
\nonumber
&E_n \downarrow E \mbox{ implies } H(E_n) \downarrow H(E)
.&
\end{eqnarray}

%s2.5 #&#
\subsection{Boolean algebras}
\label{2e}

Every finite Boolean algebra $ b $ has $ 2^n $ elements, where~$ n $
is the number of the atoms $ a_1,\ldots,a_n $ of $ b $; these atoms
satisfy $ a_k \wedge a_l = 0_b $ for $ k \ne l $, and $
a_1 \vee\cdots\vee a_n = 1_b $. All elements of $ b $ are of the
form
%
%
%e2.17 #&#
%
\begin{equation}
\label{2e1} a_{i_1} \vee\cdots\vee a_{i_k},\qquad 1 \le
i_1 < \cdots< i_k \le n.
\end{equation}
We denote by $ \Atoms(b) $ the set of all atoms of $ b $ and rewrite
\eqref{2e1} as
%
%
%e2.18 #&#
%
\begin{equation}
\label{2e2} \forall x \in b\qquad  x = \bigvee_{a\in\Atoms(b), a \le x} a.
\end{equation}

%
%fa2.29 #&#
\begin{fact}\label{2e3}
Let $ B $ be a Boolean algebra, $ b_1,b_2 \subset B $ two finite
Boolean subalgebras and $ b \subset B $ the Boolean subalgebra
generated by $ b_1 $, $ b_2 $. Then $ b $ is finite. If $ a_1 \in
\Atoms(b_1) $, $ a_2 \in\Atoms(b_2) $ and $ a_1 \wedge a_2 \ne0_B $,
then $ a_1 \wedge a_2 \in\Atoms(b) $, and all atoms of $ b $ are of
this form.
\end{fact}

\begin{pf*}{Hint}
These $ a_1 \wedge a_2 $ are the atoms of some finite Boolean
subalgebra $ b_3 $; note that $ b_1 \subset b_3 $ and $ b_2 \subset
b_3 $, but also $ b_3 \subset b $.
\end{pf*}

%
%fa2.30 #&#
\begin{fact}\label{2e4}
The following four conditions on a Boolean algebra $ B $ are
equivalent:
\begin{eqnarray*}
&\displaystyle\sup_n x_n \mbox{ exists for all }
x_1,x_2,\ldots\in B;&
\\
&\displaystyle\inf_n x_n \mbox{ exists for all }
x_1,x_2,\ldots\in B;&
\\
&\displaystyle\sup_n x_n \mbox{ exists for all }
x_1,x_2,\ldots\in B \mbox{ satisfying } x_1
\le x_2 \le\cdots;&
\\
&\displaystyle\inf_n x_n \mbox{ exists for all }
x_1,x_2,\ldots\in B \mbox{ satisfying } x_1
\ge x_2 \ge\cdots.&
\end{eqnarray*}
\end{fact}

\begin{pf*}{Hint}
First, $ \inf_n x_n = ( \sup_n x'_n )' $; second, $ \sup_n x_n
= \sup_n (x_1\vee\cdots\vee x_n) $.
\end{pf*}

A Boolean algebra $ B $ satisfying these equivalent conditions is
called \emph{$\sigma$-complete} (in other words, a Boolean \mbox{$\sigma$-algebra}).

%
%fa2.31 #&#
\begin{fact}[(\cite{Ha}, \textup{Section~14, Lemma 1})]\label{2e5}
The following two conditions on a Boolean algebra $ B $ are
equivalent:
\begin{longlist}[(a)]
\item[(a)] no uncountable subset $ X \subset B $ satisfies $ x
\wedge y = 0_B
$ for all $ x,y \in B $ (``the~\emph{countable chain condition}'');

\item[(b)] every subset $ X $ of $ B $ has a countable subset $ Y $
such that
$ X $ and $ Y $ have the same set of upper bounds.
\end{longlist}
\end{fact}

%
%fa2.32 #&#
\begin{fact}[(\cite{Ha}, \textup{Section~14, Corollary})]\label{2e6}
If a $\sigma$-complete Boolean algebra satisfies the countable chain
condition,
then it is complete.
\end{fact}

\begin{pf*}{Hint}
Use Fact \ref{2e5}(b).
\end{pf*}

%s2.6 #&#
\subsection{Measurable functions and equivalence classes}
\label{2f}

Let $ (S,\Si,\mu) $ be a measure space, $ \mu(S) < \infty$. As usual,
we often treat equivalence classes of measurable functions on $ S $ as
just measurable functions, which is harmless as long as only countably
many equivalence classes are considered simultaneously. Otherwise,
dealing with uncountable sets of equivalence classes, we must be
cautious.

All equivalence classes of measurable functions $ S \to[0,1] $ are a
complete lattice. Let $ Z $ be some set of such classes. If $ Z $ is
countable, then its supremum, $ \sup Z $, may be treated naively (as
the pointwise supremum of functions). For an uncountable $ Z $ we have
$ \sup Z = \sup Z_0 $ for some countable $ Z_0 \subset Z $. In
particular, the equality holds whenever $ Z_0 $ is dense in $ Z $
according to the $ L_1 $ metric.

The same holds for functions $ S \to\{0,1\} $ or, equivalently,
measurable sets. Functions $ S \to[0,\infty] $ are also a complete
lattice, since $ [0,\infty] $ can be transformed into $ [0,1] $ by an
increasing bijection.

In the context of \eqref{2d9}, \eqref{2d10} we have
%
%
%e2.19 #&#
%
\begin{equation}
\label{2f05} H \Bigl( \inf_{i\in I} E_i \Bigr) =
\bigcap_{i\in I} H(E_i)
\end{equation}
for an arbitrary (not just countable) family of equivalence classes $
E_i $ of measurable sets. Similarly,
%
%e2.20 #&#
%
\begin{equation}
\label{2f2} H \Bigl( \sup_{i\in I} E_i \Bigr) = \sup
_{i\in I} H(E_i),
\end{equation}
the closure of the sum of all $ H(E_i) $.

%
%fa2.33 #&#
\begin{fact}\label{2f3}
For every increasing sequence of measurable functions $ f_n\dvtx\break  S \to
[0,\infty) $ there exist $ n_1 < n_2 < \cdots$ such that almost every
$ s \in S $ satisfies one of two incompatible conditions:
\[
\mbox{either } \lim_n f_n(s) < \infty\mbox{ or } f_{n_k}(s) \ge k
\mbox{ for all $ k $ large enough}
\]
[here ``$ k $ large enough'' means $ k \ge k_0(s) $].
\end{fact}

\begin{pf*}{Hint}
Take $ n_k $ such that
\[
\sum_k \mu\Bigl( \{ s\dvtx f_{n_k} < k \}
\cap\Bigl\{ s\dvtx\lim_n f_n(s) = \infty\Bigr\}
\Bigr) < \infty.
\]
\upqed\end{pf*}

All said above holds also for a measure class space $ (S,\Si,\M) $
(see Section~\ref{2d}) in place of the measure space $ (S,\Si,\mu)
$.

%s3 #&#
\section{\texorpdfstring{Convergence of $\sigma$-fields and independence}
{Convergence of sigma-fields and independence}}
\label{sec3}

Throughout this section $ (\Om,\F,P), \La, H $ and $ Q_x $ are as in
Section~\ref{sec2}.

%s3.1 #&#
\subsection{Definition of the convergence}
\label{3a}

The strong operator topology on the projection operators $ Q_x $
induces a topology on $ \La$;\vadjust{\goodbreak} we call it the strong operator topology
on $ \La$. It is metrizable (since the strong operator topology is
metrizable on operators of norm $ \le1 $; see \cite{Co}, Section~8, Exercise
1). Thus, for $ x, x_1, x_2, \ldots\in\La$,
\[
x_n \to x \mbox{ means } \forall f \in H \| Q_{x_n} f -
Q_x f \| \mathop{\longrightarrow}^{n\to\infty}0.
\]

On the other hand we have the monotone convergence derived from the
partial order on $ \La$,
\begin{eqnarray*}
&\displaystyle x_n \downarrow x \mbox{ means } x_1 \ge x_2
\ge\cdots\mbox{ and } \inf_n x_n = x,&
\\
&\displaystyle x_n \uparrow x \mbox{ means } x_1 \le x_2 \le
\cdots\mbox{ and } \sup_n x_n = x.&
\end{eqnarray*}

By Fact \ref{2b25},
%
%
%e3.1 #&#
%
\begin{equation}
\label{3a1} x_n \downarrow x \mbox{ implies } x_n \to x;
\mbox{also, } x_n \uparrow x \mbox{ implies } x_n \to x.
\end{equation}

%s3.2 #&#
\subsection{\texorpdfstring{Commuting $\sigma$-fields}
{Commuting sigma-fields}}
\label{3b}

%
%de3.1 #&#
\begin{definition}\label{3b1}
Elements $ x,y \in\La$ are \emph{commuting},
if $ Q_x Q_y = Q_y Q_x $. A~subset of $ \La$ is \emph{commutative},
if its elements are pairwise commuting.
\end{definition}

By \eqref{2a6},
%
%
%e3.2 #&#
%
\begin{equation}
\label{3b2} \mbox{every linearly ordered subset of $ \La$ is commutative.}
\end{equation}
By Fact \ref{2b2},
%
%
%e3.3 #&#
%e3.4 #&#
%e3.5 #&#
%
\begin{eqnarray}
\label{3b3} &&\mbox{if } x_n \to x, y_n \to y,
\nonumber\\
&&\qquad\mbox{and for every $n$ the two elements } x_n, y_n
\mbox{ are commuting},
\\
&&\qquad\mbox{then } x,y \mbox{ are commuting.}\nonumber
\end{eqnarray}
In particular,
%
%
%e3.6 #&#
%
\begin{equation}
\label{3b4} \mbox{the closure of a commutative set is commutative.}
\end{equation}
It follows from Fact \ref{2c5}, or just \eqref{2a8}, that
%
%
%e3.7 #&#
%
\begin{equation}
\label{3b5} \mbox{if } x,y \in\La\mbox{ are commuting then } Q_x
Q_y = Q_{x\wedge y}.
\end{equation}

Recall $ \liminf_n x_n $ for $ x_n \in\La$ defined in Section~\ref{1c}.

%
%le3.2 #&#
\begin{lemma}\label{3b6}
If $ x_n \in\La$ are pairwise commuting and $ x_n \to x $, then\break
$ \liminf_k x_{n_k} = x $ for some $ n_1 < n_2 < \cdots$.
\end{lemma}

\begin{pf}
The commuting projection operators $ Q_{x_n} $ generate a commutative
von Neumann algebra; by Fact \ref{2d1} this algebra is isomorphic to the
algebra $ L_\infty$ on some measure space (of finite
measure). Denoting the isomorphism by\vadjust{\goodbreak} $ \al$ we have $ \al(Q_{x_n}) =
\One_{E_n} $, $ \al(Q_x) = \One_E $ (indicators of some measurable
sets $ E_n, E $). Using \eqref{3b5} we get
\[
\al( Q_{x_m \wedge x_n} ) = \One_{E_m \cap E_n}
\]
for all $ m,n $; the same holds for more than two indices.

The strong convergence of operators $ Q_{x_n} \to Q_x $ implies by
Fact \ref{2d2} convergence in measure of indicators, $ \One_{E_n} \to
\One_E $. We choose a subsequence convergent almost everywhere, $
\One_{E_{n_k}} \to\One_E $, then $ \liminf_k \One_{E_{n_k}} = \One_E
$, that is,
\[
\sup_k \inf_i \One_{E_{n_{k+i}}} =
\One_E.
\]
We have $ \al( Q_{x_{n_k} \wedge x_{n_{k+1}} \wedge\cdots\wedge
x_{n_{k+i}}} ) = \One_{E_{n_k} \cap E_{n_{k+1}} \cap\cdots\cap
E_{n_{k+i}}} $, therefore (for $ i \to\infty$), $
\al( Q_{\inf_i x_{n_{k+i}}} ) = \inf_i \One_{E_{n_{k+i}}} $, and
further (for $ k \to\infty$), $ \al( Q_{ \sup_k \inf_i x_{n_{k+i}}}
) =  \sup_k \inf_i \One_{E_{n_{k+i}}} $. We get $ \al( Q_{\liminf_k
x_{n_k}} ) = \liminf_k \One_{E_{n_k}} = \One_E = \al(Q_x) $, therefore
$ \liminf_k x_{n_k} = x $.
\end{pf}

%
%pr3.3 #&#
\begin{proposition}\label{3b7}
Assume that a set $ B \subset\La$ is commutative, and $ x \wedge y
\in B $ for all $ x,y \in B $. Then the set
\[
\Cl(B) = \Bigl\{ \liminf_n x_n\dvtx x_1,x_2,\ldots\in B \Bigr\}
\]
(lower limits of all sequences of elements of $ B $) is equal to the
topological closure of $ B $.
\end{proposition}

\begin{pf}
On one hand, if $ x_n \to x $, then $ x \in\Cl(B) $ by Lemma \ref
{3b6}. On
the other hand, $ \liminf x_n = \sup_n \inf_k x_{n+k} $ belongs to the
topological closure by~\eqref{3a1}.
\end{pf}

%
%pr3.4 #&#
\begin{proposition}\label{3b8}
Let $ x_n, y_n, x, y \in\La$, $ x_n \to x $, $ y_n \to y $, and for
each $ n $ (separately), $ x_n, y_n $ commute. Then $ x_n \wedge y_n
\to x \wedge y $.
\end{proposition}

\begin{pf}
By \eqref{3b3}, $ Q_x Q_y = Q_y Q_x $. By \eqref{3b5}, $ Q_{x\wedge
y} = Q_x Q_y $. Similarly, $ Q_{x_n \wedge y_n} = Q_{x_n} Q_{y_n}
$. Using Fact \ref{2b1} we get $ Q_{x_n \wedge y_n} \to Q_{x\wedge y} $,
that is, $ x_n \wedge y_n \to x \wedge y $.
\end{pf}

%s3.3 #&#
\subsection{\texorpdfstring{Independent $\sigma$-fields}
{Independent sigma-fields}}
\label{3c}

%
%pr3.5 #&#
\begin{proposition}\label{3c1}
The following two conditions on $ x,y \in\La$ are equivalent:
\begin{longlist}[(a)]
\item[(a)] $ x,y $ are independent;

\item[(b)] $ x,y $ are commuting, and $ x \wedge y = 0_\La$.
\end{longlist}
\end{proposition}

\begin{pf}
(a) \imp(b): independence of $ x,y $ implies $ \mathbb{E}(f |y ) =
\Ex f $ for all $ f \in L_2(x) $, that is, $ Q_y f = \langle f,\One
\rangle\One$ for $ f \in H_x $, and therefore $ Q_y Q_x = Q_0 = Q_x
Q_y $; use \eqref{3b5}.\vadjust{\goodbreak}

(b) \imp(a): by \eqref{3b5}, $ Q_y Q_x = Q_0 = Q_x Q_y $; thus $ Q_y
f = \langle f,\One\rangle\One$ for $ f \in H_x $, and therefore $
P ( A
\cap B ) = \langle\One_A,\One_B \rangle= \langle\One_A,Q_y
\One_B \rangle
= \langle Q_y \One_A,\One_B \rangle= \langle\One_A,\One
\rangle\times \langle\One,\One_B \rangle
= P(A) P(B) $ for all $ A \in x $, $ B \in y $.
\end{pf}

It may happen that $ x \wedge y = 0 $ but $ x,y $ are not
commuting. (In particular, it may happen that $ x,y $ are independent
w.r.t. some measure equivalent to $ P $, but not w.r.t. $ P $.)

%
%co3.6 #&#
\begin{corollary}
If $ x_n \to x $, $ y_n \to y $, and $ x_n,y_n $ are independent for
each~$ n $ (separately), then $ x,y $ are independent.
\end{corollary}

\begin{pf}
By Proposition \ref{3c1}, $ x_n,y_n $ are commuting, and $ x_n \wedge
y_n = 0_\La$.
By~\eqref{3b3}, $ x,y $ are commuting. By Proposition \ref{3b8}, $ x
\wedge y =
0_\La$. By Proposition~\ref{3c1} (again), $ x,y $ are independent.
\end{pf}

%s3.4 #&#
\subsection{\texorpdfstring{Product $\sigma$-fields}
{Product sigma-fields}}
\label{3d}

For every $ x \in\La$ the triple $ (\Om,x,P|_x) $ is also a
probability space, and it may be used similarly to $ (\Om,\F,P) $,
giving the complete lattice $ \La(\Om,x,P|_x) $, endowed with the
topology, etc. This lattice is naturally embedded into $ \La$,
\[
\La(\Om,x,P|_x) = \{ y \in\La\dvtx y \le x \}.
\]
The lattice operations ($ \wedge$, $ \vee$), defined on $
\La(\Om,x,P|_x) $, do not differ from these induced from $ \La$
(which is evident); also the topology, defined on $ \La(\Om,x,P|_x) $,
does not differ from the topology induced from $ \La$ (which follows
easily from the equality $ Q_y = Q_y^{(x)} Q_x $ for $ y \le x $; see
Notation \ref{2c2} for~$ Q_y^{(x)} $). Thus it is correct to define $
\La_x $,
as a lattice and topological space,\footnote{%
Not ``topological lattice'' since the lattice operations are
generally not continuous.}
by
\[
\La(\Om,x,P|_x) = \La_x = \{ y \in\La\dvtx y \le x \}
\subset\La.
\]

Given $ x,y \in\La$, the product set $ \La_x \times\La_y $
carries the product topology and the product partial order, and is
again a lattice (see \cite{DP}, Section~2.15, for the product of two
lattices), moreover, a complete lattice (see \cite{DP}, Exercise
2.26(ii)).

On the other hand, for independent $ x,y \in\La$ we introduce
\[
\La_{x,y} = \{ u \vee v\dvtx u \le x, v \le y \} \subset
\La_{x\vee y}.
\]
Generally, $ \La_{x,y} $ is only a small part of $ \La_{x\vee y} $;
indeed, a sub-$\sigma$-field on the product of two probability spaces is
generally not a product of two sub-$\sigma$-fields. This fact is a
manifestation of nondistributivity of the lattice $ \La$; the
equality
\[
( x \wedge z ) \vee( y \wedge z ) = ( x \vee y ) \wedge z
\]
fails whenever $ z \in\La_{x\vee y} \setminus\La_{x,y} $.

%
%le3.7 #&#
\begin{lemma}\label{3d1}
Every element of $ \La_{x,y} $ is commuting with $ x $ (and $ y $).
\end{lemma}

\begin{pf}
By Fact \ref{2c3}, treating $ H_{x\vee y} $ as $ H_x \otimes H_y $ we have
$ Q_{u\vee v} = Q_u^{(x)} \otimes Q_v^{(y)} $ whenever $ u \le x $, $
v \le y $. Also, $ Q_x = Q_x^{(x)} \otimes Q_0^{(y)} $. By
\eqref{3b2}, $ Q_u^{(x)} $ and $ Q_x^{(x)} $ are commuting; the same
holds for $ Q_v^{(y)} $ and $ Q_0^{(y)} $. Therefore $ Q_{u\vee v} $
and $ Q_x $ are commuting.
\end{pf}

%
%th3.8 #&#
\begin{theorem}\label{3d2}
If $ x,y \in\La$ are independent, then $ \La_{x,y} $ is a closed
subset of~$ \La$, the maps
\begin{eqnarray*}
&\La_x \times\La_y \ni(u,v) \mapsto u \vee v \in
\La_{x,y},&
\\
&\La_{x,y} \ni z \mapsto( x \wedge z, y \wedge z ) \in
\La_x \times\La_y&
\end{eqnarray*}
are mutually inverse bijections, and each of them is both an
isomorphism of lattices and a homeomorphism of topological spaces.
\end{theorem}

\begin{pf}
The composition map $ \La_x \times\La_y \to\La_{x,y} \to\La_x
\times\La_y $ is the identity by \eqref{2c8}. Taking into account
that the map $ \La_x \times\La_y \to\La_{x,y} $ is surjective we get
mutually inverse bijections.

The map $ \La_x \times\La_y \to\La_{x,y} $ preserves lattice
operations: ``$\wedge$'' by Fact \ref{2c7}, and ``$\vee$''
trivially. It is
a bijective homomorphism, therefore, isomorphism of lattices.

Let $ u,u_1,u_2,\ldots\in\La_x $, $ u_n \to u $, and $
v,v_1,v_2,\ldots\in\La_y $, $ v_n \to v $. Then $ Q_{u_n}^{(x)} \to
Q_u^{(x)} $ and $ Q_{v_n}^{(y)} \to Q_v^{(y)} $. By Fact \ref{2b5}, $
Q_{u_n}^{(x)} \otimes Q_{v_n}^{(y)} \to Q_u^{(x)} \otimes Q_v^{(y)}
$. By Fact \ref{2c3}, $ Q_{u_n\vee v_n} \to Q_{u\vee v} $, that is, $ u_n
\vee v_n \to u \vee v $. The map $ \La_x \times\La_y \to\La_{x,y} $
is thus continuous.

Let $ z_1,z_2,\ldots\in\La_{x,y} $, $ z_n \to z \in\La$. By
Lemma \ref{3d1} and Proposition \ref{3b8}, $ x \wedge z_n \to x
\wedge z $. Similarly, $
y \wedge z_n \to y \wedge z $. In particular, taking $ z \in\La_{x,y}
$ we see that the map $ \La_{x,y} \to\La_x \times\La_y $ is
continuous. In general (for $ z \in\La$) we get $ z_n = ( x \wedge
z_n ) \vee( y \wedge z_n ) \to( x \wedge z ) \vee( y \wedge z ) $,
therefore $ z = ( x \wedge z ) \vee( y \wedge z ) \in\La_{x,y} $; we
see that $ \La_{x,y} $ is closed.
\end{pf}

It follows that
%
%
%e3.8 #&#
%
\begin{equation}
\label{3d3} \La_{x,y} = \bigl\{ z \in\La\dvtx z = ( x \wedge z )
\vee( y
\wedge z ) \bigr\}.
\end{equation}

%
%re3.9 #&#
\begin{remark}\label{3d4}
By Theorem \ref{3d2}, any relation between elements of $ \La_{x,y} $ expressed
in terms of lattice operations (and limits) is equivalent to the
conjunction of two similar relations ``restricted'' to $ x $ and $ y
$. For example, the relation
\[
( z_1 \vee z_2 ) \wedge z_3 =
z_4 \vee z_5
\]
between $ z_1, z_2, z_3, z_4, z_5 \in\La_{x,y} $ splits in two;
first,
\[
\bigl( ( x \wedge z_1 ) \vee( x \wedge z_2 ) \bigr)
\wedge( x \wedge z_3 ) = ( x \wedge z_4 ) \vee( x \wedge
z_5 ),
\]
and second, a similar relation with $ y $ in place of $ x $.
\end{remark}

%s4 #&#
\section{Noise-type completion}
\label{sec4}

Throughout Sections~\ref{sec4}--\ref{sec7}, $ B \subset\La$ is a
noise-type Boolean algebra (as defined by Definition \ref{1a1}); $ \La
$, $ H $
and $ Q_x $ are as in Section~\ref{sec2}.

%s4.1 #&#
\subsection{\texorpdfstring{The closure; proving Theorem \protect\ref{1d1}}
{The closure; proving Theorem 1.6}}
\label{4a}

By separability of $ H $,
%
%
%e4.1 #&#
%
\begin{equation}
\label{4a03} B \mbox{ satisfies the countable chain condition},
\end{equation}
since otherwise there exists an uncountable set of pairwise orthogonal
nontrivial subspaces of $H$. By Fact \ref{2e6},
%
%
%e4.2 #&#
%
\begin{equation}
\label{4a05} B \mbox{ is complete if and only if it is $\sigma$-complete.}
\end{equation}

Recall that every $ x \in B $ has its complement $ x' \in B $,
\[
x \wedge x' = 0_\La,\qquad x \vee x' =
1_\La; \qquad x,x' \mbox{ are independent.}
\]
(The complement in $ B $ is unique, however, many other independent
complements may exist in $ \La$.)

By distributivity of $ B $, $ y = ( x \wedge y ) \vee( x' \wedge y )
$ for all $ x,y \in B $; by \eqref{3d3},
%
%
%e4.3 #&#
%
\begin{equation}
\label{4a1} B \subset\La_{x,x'} \qquad\mbox{for every } x \in B.
\end{equation}
By Lemma \ref{3d1},
%
%
%e4.4 #&#
%
\begin{equation}
\label{4a2} B \mbox{ is a commutative subset of } \La.
\end{equation}
Recall $ \Cl(B) $ introduced in Theorem \ref{1d1}; by Proposition
\ref{3b7},
%
%
%e4.5 #&#
%
\begin{equation}
\label{4a3} \mbox{the topological closure of $B$ is } \Cl(B) = \Bigl
\{ \liminf
_n x_n\dvtx x_1,x_2,\ldots
\in B \Bigr\}.
\end{equation}
Taking into account that $ \La_{x,x'} $ is closed by Theorem \ref{3d2},
we get from \eqref{4a1}
%
%
%e4.6 #&#
%
\begin{equation}
\label{4a4} \Cl(B) \subset\La_{x,x'}\qquad \mbox{for every } x \in B.
\end{equation}
By \eqref{4a2} and \eqref{3b4},
%
%
%e4.7 #&#
%
\begin{equation}
\label{4a5} \Cl(B) \mbox{ is a commutative subset of } \La.
\end{equation}
By Proposition \ref{3b8},
%
%
%e4.8 #&#
%
\begin{equation}
\label{4a6} x \wedge y \in\Cl(B) \qquad\mbox{for all } x,y \in\Cl(B).
\end{equation}
By \eqref{3b5},
%
%
%e4.9 #&#
%
\begin{equation}
\label{4a7} Q_x Q_y = Q_{x\wedge y} \qquad\mbox{for all }
x,y \in\Cl(B).
\end{equation}

\begin{pf*}{Proof of Theorem \ref{1d1}}
If $ x_n \in\Cl(B) $ and $ x_n \uparrow x $, then $ x_n \to x $ by
\eqref{3a1}, therefore $ x \in\Cl(B) $, which proves item (b) of the
theorem.

If $ x_n \in\Cl(B) $ and $ x = \inf_n x_n $, then $ x_1 \wedge\cdots
\wedge x_n = y_n \in\Cl(B) $ by \eqref{4a6} and $ y_n \downarrow x $,
thus $ y_n \to x $ by \eqref{3a1} (again) and $ x \in\Cl(B) $, which
proves item (a) of the theorem.
\end{pf*}

By Proposition \ref{3c1} and \eqref{4a5}, for $ x,y \in\Cl(B) $,
%
%
%e4.10 #&#
%
\begin{equation}
\label{4a8} x \wedge y = 0_\La\mbox{ if and only if } x,y \mbox{ are
independent.}\vadjust{\goodbreak}
\end{equation}

By Proposition \ref{3b8} and \eqref{4a5}, for $ x,x_n,y,y_n \in\Cl
(B) $,
%
%
%e4.11 #&#
%
\begin{equation}
\label{4a9} \mbox{if } x_n \to x, y_n \to y \mbox{ then }
x_n \wedge y_n \to x \wedge y.
\end{equation}

%
%re4.1 #&#
\begin{remark}\label{4a10}
In contrast, $ x_n \vee y_n $ need not converge to $ x \vee y $, even
if $ x_n \in B $, $ x_n \downarrow0_\La$, $ y_n = x'_n $; it may
happen that $ y_n \uparrow y $, $ y \ne1_\La$. This situation
appears already in the (simplest nonclassical) example given in
Section~\ref{1b}.

On the other hand, if $ x_n \in B $, $ x_n \to1_\La$, then
necessarily $ x'_n \to0_\La$ (but we do not need this fact).
\end{remark}

By Theorem \ref{3d2}, for every $ z \in B $ the map $ x \mapsto x
\wedge z $ is a lattice homomorphism $ \La_{z,z'} \to\La_z $, thus, $
( x \vee y ) \wedge z = ( x \wedge z ) \vee( y \wedge z ) $ for all $
x,y \in\La_{z,z'} $; in particular, it holds for all $ x,y \in\Cl(B)
$ by \eqref{4a4}. If $ x \vee y = 1_\La$, then $ z = ( x \wedge z )
\vee( y \wedge z ) $. If in addition $ x \wedge y = 0_\La$, then $
x,y $ are independent by \eqref{4a8}, and $ z \in\La_{x,y} $ by
\eqref{3d3}. Thus $ B \subset\La_{x,y} $. By Theorem \ref{3d2} $
\La_{x,y} $ is closed, and we conclude.

%
%pr4.2 #&#
\begin{proposition}\label{4a11}
If $ x,y \in\Cl(B) $, $ x \wedge y = 0_\La$, $ x \vee y = 1_\La$,
then\break  \mbox{$ \Cl(B) \subset\La_{x,y} $}.
\end{proposition}

%
%co4.3 #&#
\begin{corollary}\label{4a12}
For every $ x \in\Cl(B) $ there exists at most one $ y \in\Cl(B) $
such that $ x \wedge y = 0_\La$ and $ x \vee y = 1_\La$.
\end{corollary}

\begin{pf}
Assume that $ y_1,y_2 \in\Cl(B) $, $ x \wedge y_k = 0_\La$ and $ x
\vee y_k = 1_\La$ for $ k=1,2 $. By Proposition \ref{4a11}, $ y_2 \in
\La_{x,y_1}
$, that is, $ y_2 = ( x \wedge y_2 ) \vee( y_1 \wedge y_2 ) = y_1
\wedge y_2 $. Similarly, $ y_1 = y_2 \wedge y_1 $.
\end{pf}

%s4.2 #&#
\subsection{\texorpdfstring{The completion; proving Theorem \protect\ref{1d2}}
{The completion; proving Theorem 1.7}}
\label{4b}

Let $ B $ and $ \Cl(B) $ be as in Section~\ref{4a}, and
\[
C = \bigl\{ x \in\Cl(B)\dvtx\exists y \in\Cl(B) \> x \wedge y =
0_\La,
x \vee y = 1_\La\bigr\}
\]
as in Theorem \ref{1d2}; clearly,
%
%
%e4.12 #&#
%
\begin{equation}
\label{4b1} B \subset C \subset\Cl(B).
\end{equation}
Taking Corollary \ref{4a12} into account, we extend the complement
operation, $
x \mapsto x' $, from $ B $ to $ C $:
\begin{eqnarray*}
&x' \in C \mbox{ for } x \in C; \qquad\bigl(x'
\bigr)' = x;&
\\
&x \wedge x' = 0_\La; \qquad x \vee x' =
1_\La.&
\end{eqnarray*}
By \eqref{4a8}, $ x, x' $ are independent; and by Proposition \ref{4a11},
%
%
%e4.13 #&#
%
\begin{equation}
\label{4b2} \forall x \in C\qquad  \Cl(B) \subset\La_{x,x'}.
\end{equation}

%
%le4.4 #&#
\begin{lemma}\label{4b3}
For every $ x \in C $ the map
\[
\Cl(B) \ni y \mapsto x \vee y \in\La
\]
is continuous.\vadjust{\goodbreak}
\end{lemma}

\begin{pf}
Let $ y_n,y \in\Cl(B) $, $ y_n \to y $; we have to prove that $ x
\vee y_n \to x \vee y $. By \eqref{4a9}, $ x' \wedge y_n \to x'
\wedge y $. Applying Theorem \ref{3d2} to $ ( x, x' \wedge y_n ) \in
\La_x \times\La_{x'} $ we get $ x \vee( x' \wedge y_n ) \to x \vee(
x' \wedge y ) $. It remains to prove that $ x \vee( x' \wedge y_n ) =
x \vee y_n $ and $ x \vee( x' \wedge y ) = x \vee y $. We prove the
latter; the former is similar. Note that $ y \in\Cl(B) \subset
\La_{x,x'} $ by \eqref{4b2}. The lattice isomorphism $ \La_{x,x'}
\to
\La_x \times\La_{x'} $ of Theorem~\ref{3d2} maps $ x $ into $ (x,0) $
and $ y $ into $ ( x \wedge y, x' \wedge y ) $; therefore it maps $ x
\vee y $ into $ ( x \vee( x \wedge y ), 0 \vee( x' \wedge y ) )
= (x, x' \wedge y) $, which implies $ x \vee( x' \wedge y ) = x \vee
y $.
\end{pf}

%
%le4.5 #&#
\begin{lemma}\label{4b4}
\[
\forall x \in C\ \forall y \in\Cl(B)\qquad x \vee y \in\Cl(B).
\]
\end{lemma}

\begin{pf}
By Lemma \ref{4b3} it is sufficient to consider $ y \in B $. Applying
Lemma \ref{4b3} (again) to $ y \in B \subset C $ we see that the map $
\Cl(B) \ni z \mapsto y \vee z \in\La$ is continuous. This map sends
$ B $ into $ B $, and therefore it sends $ x \in C \subset\Cl(B) $
into $
\Cl(B) $.
\end{pf}

%
%le4.6 #&#
\begin{lemma}\label{4b5}
For all $ x,y \in C $,
\[
x \vee y \in C \quad\mbox{and}\quad ( x \vee y )' = x' \wedge
y'.
\]
\end{lemma}

\begin{pf}
By Lemma \ref{4b4}, $ x \vee y \in\Cl(B) $. By \eqref{4a6}, $ x'
\wedge y'
\in\Cl(B) $. We have to prove that $ ( x \vee y ) \wedge( x' \wedge
y' ) = 0_\La$ and $ ( x \vee y ) \vee( x' \wedge y' ) = 1_\La$. We
do it using Remark \ref{3d4}.

First, $ x,y,x',y' \in C \subset\Cl(B) \subset\La_{x,x'} $.

Second, we consider $ z = ( x \vee y ) \wedge( x' \wedge y' ) $ and
``restrict'' it first to $ x $: $ x \wedge z = ( x \vee( x \wedge y )
) \wedge( 0_\La\wedge( x \wedge y' ) ) = 0_\La$, and second, to $
x' $: $ x' \wedge z = ( 0_\La\vee( x' \wedge y ) ) \wedge x' \wedge
( x' \wedge y' ) \le y \wedge y'= 0_\La$. We get $ z = 0_\La$, that
is, $ ( x \vee y ) \wedge( x' \wedge y' ) = 0_\La$.

Third, we consider $ z = ( x \vee y ) \vee( x' \wedge y' ) $ and get
$ x \wedge z = x \vee( x \wedge y ) \vee( x \wedge x' \wedge y' ) =
x $ and $ x' \wedge z = ( x' \wedge x ) \vee( x' \wedge y ) \vee( x'
\wedge x' \wedge y' ) = ( x' \wedge y ) \vee( x' \wedge y' ) = x'
\wedge( y \vee y' ) = x' $. Therefore $ z = x \vee x' = 1_\La$, that
is, $ ( x \vee y ) \vee( x' \wedge y' ) = 1_\La$.
\end{pf}

In addition, $ x \wedge y = (x' \vee y')' \in C $ for all $ x,y \in C
$; thus $ C $ is a sublattice of $ \La$. The lattice $ C $ is
distributive, that is, $ x \wedge( y \vee z ) = ( x \wedge y ) \vee(
x \wedge z ) $ for all $ x,y,z \in C $, since $ C \subset\La_{x,x'} $
by \eqref{4b1}, \eqref{4b2}, and the map $ \La_{x,x'} \ni y \mapsto x
\wedge y \in\La_x $ is a lattice homomorphism by Theorem
\ref{3d2}. Also, $ 0_\La\in C $, $ 1_\La\in C $, and each $ x \in C
$ has a complement $ x' $ in $ C $. By \eqref{4b1} and \eqref{4a8}, $
x,x' $ are independent for every $ x \in C $. Thus $ C $ is a
noise-type Boolean algebra satisfying \eqref{4b1}, which proves item
(a) of Theorem \ref{1d2}.

If $ C_1 $ is also a noise-type Boolean algebra satisfying $ B \subset
C_1 \subset\Cl(B) $, then every element of $ C_1 $ belongs to $ C $,
since its complement in $ C_1 $ is also its complement in $ \Cl(B)
$. Thus $ C_1 \subset C $, which proves item (b) of Theorem~\ref{1d2}.

%
%co4.7 #&#
\begin{corollary}\label{4b6}
The following two conditions on a noise-type Boolean algebra $ B $ are
equivalent:
\begin{longlist}[(a)]
\item[(a)] $ C = \Cl(B) $ (where $ C $ is the completion of $ B $);

\item[(b)] there exists a complete noise-type Boolean algebra $ \hat
B $ such
that $ B \subset\hat B $.
\end{longlist}
\end{corollary}

\begin{pf}
(a) \imp(b): the noise-type Boolean algebra $ C = \Cl(B) $ is closed;
by~\eqref{3a1} it is $\sigma$-complete (recall Section~\ref{2e});
by \eqref{4a05}
it is complete.

(b) \imp(a): Given $ x \in\Cl(B) $, we take $ x_n \in B $ such that $
x = \liminf_n x_n $ [recall~\eqref{4a3}]; $ x \in\hat B $. The
complement $ x' $ of $ x $ in $ \hat B $ belongs to $ \Cl(B) $, since
$ ( \liminf_n x_n )' = \limsup_n x'_n $ in $ \hat B $. Thus, $ x $ is
complemented in $ \Cl(B) $, that is, $ x \in C $.
\end{pf}

%s5 #&#
\section{Classicality and blackness}
\label{sec5}
% \input{sect5}

%s5.1 #&#
\subsection{Atomless algebras}
\label{5a}

Recall Section~\ref{1e}.

%
%pr5.1 #&#
\begin{proposition}\label{5a1}
If $ B $ is atomless, then for every $ f \in H $ satisfying $ Q_0 f = 0
$ and $ \eps> 0 $ there exist $ n $ and $ x_1,\ldots,x_n \in B $ such
that
\[
x_1 \vee\cdots\vee x_n = 1_\La\quad\mbox{and}\quad \|
Q_{x_1} f \| \le\eps, \ldots, \| Q_{x_n} f \| \le\eps.
\]
\end{proposition}

The proof is given after three lemmas.

%
%le5.2 #&#
\begin{lemma}\label{5a2}
Let $ F \subset B $ be a filter such that $ \inf_{x\in F} x = 0_\La
$. Then\break  $ \inf_{x\in F} \| Q_x f \| = 0 $ for all $ f \in H $
satisfying $ Q_0 f = 0 $.
\end{lemma}

\begin{pf}
Given such $ f $, we denote $ c = \inf_{x\in F} \| Q_x f \| $, assume
that $ c>0 $ and seek a contradiction.

We choose $ x_n \in F $ such that $ x_1 \ge x_2 \ge\cdots$ and $ \|
Q_{x_n} f \| \downarrow c $. Necessarily, $ x_n \downarrow x $ for
some $ x \in\La$; by \eqref{2b4}, $ Q_{x_n} \to Q_x $, thus $ \| Q_x
f \| = c $.

For arbitrary $ y \in F $ we have $ \| Q_y Q_{x_n} f \| \ge c $ [since
$ Q_y Q_{x_n} = Q_{y\wedge x_n} $ by \eqref{4a7}, and $ y \wedge x_n
\in F $], therefore $ \| Q_y Q_x f \| \ge c = \| Q_x f \| $, which
implies\break  $ Q_y Q_x f = Q_x f $, that is, $ Q_x f \in H_y $ for all $ y
\in F $. By Fact \ref{2a9}, $ \bigcap_{y\in F} H_y = H_0 $. We get $
Q_x f
\in H_0 $, $ Q_0 Q_x f = 0 $ and $ \| Q_x f \| \ne0 $; a
contradiction.
\end{pf}

%
%le5.3 #&#
\begin{lemma}\label{5a3}
Let a function $ m\dvtx B \to[0,\infty) $ satisfy $ m ( x \vee y ) +
m (
x \wedge y ) \ge m(x) + m(y) $ for all $ x,y \in B $, and $ m(0_\La) =
0 $. Then the following two conditions on $ m $ are equivalent:
\begin{longlist}[(a)]
\item[(a)] for every $ \eps> 0 $ there exist $ n $ and $ x_1,\ldots
,x_n \in B
$ such that $ x_1 \vee\cdots\vee x_n = 1_\La$ and $ m(x_1) \le\eps,
\ldots, m(x_n) \le\eps$;

\item[(b)] $ \inf_{x\in F} m(x) = 0 $ for every ultrafilter $ F
\subset B $.
\end{longlist}
\end{lemma}

\begin{pf}
(a) \imp(b): the ultrafilter must contain at least one $ x_k $, thus
$ \inf_{x\in F} m(x) \le\eps$ for every $ \eps$.\vadjust{\goodbreak}

(b) \imp(a): we assume that (a) is violated and prove that (b) is
violated.

Note that $ m ( x \vee y ) \ge m(x) + m(y) \ge m(x) $ whenever $ x
\wedge y = 0_\La$, and therefore $ m(x) \ge m(y) $ whenever $ x \ge y
$.

We define $ \ga\dvtx B \to[0,\infty) $ by
\[
\ga(x) = \inf_{x_1\vee\cdots\vee x_n=x} \max\bigl( m(x_1),
\ldots,m(x_n) \bigr),
\]
the infimum being taken over all $ n $ and all $ x_1,\ldots,x_n \in B $
such that $ x_1 \vee\cdots\vee x_n = x $. We denote $ c = \ga(1_\La)
$ and note that $ c>0 $ [since (a) is violated]. Clearly, $ \ga(x) \le
m(x) $, and $ \ga(x\vee y) = \max( \ga(x), \ga(y) ) $ for all $ x,y
\in B $.

\begin{claim*} For every $ x \in B $ and $ \eps> 0 $ there exists $ y
\in B $ such that $ y \le x $ and $ \ga(x) = \ga(y) \le m(y) \le
\ga(x)+\eps$.
\end{claim*}

\begin{pf} Take $ x_1,\ldots,x_n $ such that $ x_1
\vee\cdots\vee x_n = x $ and $ \max_k m(x_k) \le\ga(x) + \eps$;
note that $ \ga(x) = \max_k \ga(x_k)$, choose $ k $ such that $ \ga(x)
= \ga(x_k) $, and then $ y = x_k $ fits.
\end{pf}

Iterating the transition from $ x $ to $ y $ we construct $ x_0, x_1,
x_2, \ldots\in B $ such that $ 1_\La= x_0 \ge x_1 \ge x_2 \ge\cdots
$, $ \ga(x_n) = c $ for all $ n $, and $ m(x_n) \downarrow c $ as $ n
\to\infty$.

We introduce
\[
F = \Bigl\{ y \in B\dvtx\lim_n m(x_n\wedge y) \ge
c \Bigr\} = \bigl\{ y \in B\dvtx m(x_n\wedge y) \downarrow c \bigr\}
\]
and note that $ \inf_{y\in F} m(y) \ge c > 0 $ [just because $ m(y)
\ge m(x_n\wedge y) $]. It is sufficient to prove that $ F $ is an
ultrafilter.

If $ y \in F $ and $ y \le z $, then $ z \in F $ [just because $
m(x_n\wedge y) \le m(x_n\wedge z) $].

If $ y,z \in F $, then $ m(x_n) \ge m ( (x_n\wedge y) \vee(x_n\wedge
z) ) \ge m(x_n\wedge y) + m(x_n\wedge z) - m ( (x_n\wedge y) \wedge
(x_n\wedge z) ) $, therefore $ \lim_n m(x_n\wedge y\wedge z) \ge
\lim_n m(x_n\wedge y) + \lim_n m(x_n\wedge z) - \lim_n m(x_n) = c $,
thus $ y \wedge z \in F $. We conclude that $ F $ is a filter.

For arbitrary $ y \in B $ we have $ c = \ga(x_n) \le\max(
m(x_n\wedge y), m(x_n\wedge y') ) $ for all $ n $; thus $ c \le
\lim_n \max( m(x_n\wedge y), m(x_n\wedge y') ) = \max( \lim_n
m(x_n\wedge y), \lim_n m(x_n\wedge y') ) $, which shows that $ y
\notin F \impl y' \in F $. We conclude that $ F $ is an ultrafilter,
which completes the proof.
\end{pf}

%
%le5.4 #&#
\begin{lemma}
$ Q_x + Q_y \le Q_{x\vee y} + Q_{x\wedge y} $ for all $ x,y \in B $.
\end{lemma}

\begin{pf}
By \eqref{4a2}, $ Q_x $ and $ Q_y $ are commuting projections, which
implies $ Q_x + Q_y = Q_x \vee Q_y + Q_x \wedge Q_y $, where
$ Q_x \vee Q_y $ and $ Q_x \wedge Q_y $ are projections onto $ Q_x H +
Q_y H $ and $ Q_x H \cap Q_y H $, respectively. Using \eqref{4a7}, $
Q_x \wedge Q_y = Q_x Q_y = Q_{x\wedge y} $. It remains to note that $
Q_x \vee Q_y \le Q_{x\vee y} $ just because $ Q_x \le Q_{x\vee y} $
and $ Q_y \le Q_{x\vee y} $.
\end{pf}

Taking into account that $ \| Q_x \psi\|^2 = \langle Q_x \psi,\psi
\rangle
$ we get
%
%
%e5.1 #&#
%
\begin{equation}
\label{5a5} \| Q_x f \|^2 + \| Q_y f
\|^2 \le\| Q_{x\vee y} f \|^2 + \| Q_{x\wedge y} f
\|^2
\end{equation}
for all $ x,y \in B $ and $ f \in H $. Thus, the function $ m\dvtx x
\mapsto\| Q_x f \|^2 $ satisfies the condition $ m ( x \vee y ) + m (
x \wedge y ) \ge m(x) + m(y) $ of Lemma~\ref{5a3}; the other
condition, $
m(0_\La) = 0 $, is also satisfied if $ Q_0 f = 0 $.

\begin{pf*}{Proof of Proposition \ref{5a1}}
Let $ f \in H $, $ Q_0 f = 0 $. By Lemma~\ref{5a2}, $ \inf_{x\in F}
\| Q_x f
\| = 0 $ for every ultrafilter $ F \subset B $. It remains to apply
Lemma \ref{5a3} to $ m\dvtx x \mapsto\| Q_x f \|^2 $.
\end{pf*}

%s5.2 #&#
\subsection{\texorpdfstring{The first chaos; proving Proposition \protect\ref{1d4}}
{The first chaos; proving Proposition 1.10}}
\label{5b}

Let $ C $ be the completion of $ B $; see Definition \ref{1d3}. Recall the
first chaos space $ H^{(1)}(B) \subset H $ (Definition~\ref{1a2}).

%
%le5.5 #&#
\begin{lemma}\label{5b1}
The following three conditions on $ f \in H $ are equivalent:
\begin{longlist}[(a)]
\item[(a)] $ f \in H^{(1)}(B) $, that is, $ f = Q_x f + Q_{x'} f
$ for all $ x \in B $;

\item[(b)] $ Q_{x\vee y} f = Q_x f + Q_y f $ for all $ x,y \in B $
satisfying $ x \wedge y = 0_\La$;

\item[(c)] $ Q_{x\vee y} f + Q_{x\wedge y} f = Q_x f + Q_y f $
for all $ x,y \in B $, and $ Q_0 f = 0 $.
\end{longlist}
\end{lemma}

\begin{pf}
Condition (a) for $ x=0_\La$ gives $ f = Q_0 f + f $, that
is, $ Q_0 f = 0 $. Condition (b) for $ x=y=0_\La$ gives $ Q_0 f
= Q_0 f + Q_0 f $, that is, $ Q_0 f = 0 $ (again). Condition
(c) requires $ Q_0 f = 0 $ explicitly. Thus, we restrict ourselves
to $ f $ satisfying $ Q_0 f = 0 $.

Clearly, (c) \imp(b) \imp(a); we'll prove that (a) \imp(b) \imp
(c). Recall \eqref{4a7}: $ Q_x Q_y = Q_{x\wedge y} $.

(a) \imp(b): If $ x \wedge y = 0_\La$, then $ Q_{x\vee y} f =
Q_{x\vee
y} ( Q_x f + Q_{x'} f ) = Q_{x\vee y} Q_x f + Q_{x\vee y}
Q_{x'} f = Q_{(x\vee y)\wedge x} f + Q_{(x\vee y)\wedge x'} f
= Q_x f + Q_y f $.

(b) \imp(c): we apply (b) twice; first, to $ x $ and $ x' \wedge y $,
getting $ Q_{x\vee y} f = Q_x f + Q_{x'\wedge y} f $, and
second, to $ x \wedge y $ and $ x' \wedge y $, getting $ Q_y f =
Q_{x\wedge y} f + Q_{x'\wedge y} f $. It remains to eliminate $
Q_{x'\wedge y} f $.
\end{pf}

\begin{pf*}{Proof of Proposition \ref{1d4}}
It is sufficient to prove that $ H^{(1)}(B) = H^{(1)}(C) $.
The inclusion $ H^{(1)}(B) \supset H^{(1)}(C) $ follows readily from
the inclusion $ B \subset C $. We have to prove that $ H^{(1)}(B)
\subset H^{(1)}(C) $. Let $ f \in H^{(1)}(B) $. By Lemma~\ref{5b1}, $ Q_0
f = 0 $ and $ Q_{x\vee y} f + Q_{x\wedge y} f = Q_x f +
Q_y f $ for all $ x,y \in B $; it is sufficient to extend this
equality to all $ x,y \in C $. We do it in two steps: first, we extend
it to $ x \in B $, $ y \in C $ by separate continuity in $ y $ for
fixed $ x $; and second, we extend it to $ x,y \in C $ by separate
continuity in $ x $ for fixed $ y $. The separate continuity of $ x
\vee y $ is ensured by Lemma~\ref{4b3}. Continuity of $ x \wedge y $ is
ensured by \eqref{4a9}.
\end{pf*}

From now on we often abbreviate $ H^{(1)}(B) $ to $ H^{(1)} $.

\begin{claim*}
The space $ H^{(1)} $ is invariant under projections $ Q_x $ for $ x
\in B $ and moreover, for $ x \in\Cl(B) $.
\end{claim*}
\begin{pf} For $ f \in H^{(1)}
$, $ x \in\Cl(B) $ and $ g = Q_x f $ we have, using \eqref{4a5}, $
Q_y g + Q_{y'} g = ( Q_y + Q_{y'} ) Q_x f = Q_x ( Q_y + Q_{y'} ) f =
Q_x f = g $ for all $ y \in B $, which means $ g \in H^{(1)} $.
\end{pf}

We denote the restriction of $ Q_x
$ to $ H^{(1)} $ by $ Q_x^{(1)} $; using \eqref{4a7} and Lemma \ref
{5b1} we
have for all $ x,y \in B $,
%
%
%e5.2 #&#
%e5.3 #&#
%e5.4 #&#
%e5.5 #&#
%e5.6 #&#
%
\begin{eqnarray}
\qquad &Q_x^{(1)}\dvtx H^{(1)} \to H^{(1)};\qquad
Q_x^{(1)} f = Q_x f;\qquad Q_0^{(1)}
= 0, \qquad Q_1^{(1)} = I; &\label{5b2}
\\
&Q_{x\wedge y}^{(1)} = Q_x^{(1)}
Q_y^{(1)};&\label{5b3}
\\
&Q_{x\vee y}^{(1)} + Q_{x\wedge y}^{(1)} =
Q_x^{(1)} + Q_y^{(1)};&\label{5b4}
\\
&Q_{x\vee y}^{(1)} = Q_x^{(1)} +
Q_y^{(1)}\qquad \mbox{whenever } x \wedge y = 0_\La;&
\label{5b5}
\\
&Q_x^{(1)} + Q_{x'}^{(1)} = I;&
\label{5b6}
\end{eqnarray}
here $ I $ is the identity operator on $ H^{(1)} $.

%s5.3 #&#
\subsection{\texorpdfstring{Sufficient subalgebras; proving Theorem \protect\ref{1e2}}
{Sufficient subalgebras; proving Theorem 1.13}}
\label{5c}

%
%le5.6 #&#
\begin{lemma}\label{5c1}
The following two conditions on $ x \in B $ and $ f \in H $ are
equivalent:
\begin{longlist}[(a)]
\item[(a)] $ f = Q_x f + Q_{x'} f $;

\item[(b)] $ \Ex f = 0 $, and $ \Ex( f g h ) = 0 $ for all $ g
\in H_x $, $ h \in H_{x'} $ satisfying $ \Ex g = 0 $,
$ \Ex h = 0 $.
\end{longlist}
\end{lemma}

\begin{pf}
Treating $ H $ as $ H_x \otimes H_{x'} $ according to Fact \ref{2c1} we
have
\begin{eqnarray*}
H &=& \bigl( (H_x \ominus H_0) \oplus H_0
\bigr) \otimes\bigl( (H_{x'} \ominus H_0) \oplus
H_0 \bigr)
\\
&=& (H_x \ominus H_0) \otimes(H_{x'} \ominus
H_0) \oplus(H_x \ominus H_0) \otimes
H_0 \oplus H_0 \otimes(H_{x'} \ominus
H_0) \\
&&{}\oplus H_0 \otimes H_0;
\end{eqnarray*}
here $ H_x \ominus H_0 $ is the orthogonal
complement of $ H_0 $ in $ H_x $ (it consists of all zero-mean
functions of $ H_x $). In this notation $ Q_x + Q_{x'} $ becomes
\begin{eqnarray*}
&&I \otimes Q_0^{(x')} + Q_0^{(x)}
\otimes I\\
&&\qquad=\bigl( \bigl(I-Q_0^{(x)}\bigr) +
Q_0^{(x)} \bigr) \otimes Q_0^{(x')} +
Q_0^{(x)} \otimes\bigl( \bigl(I-Q_0^{(x')}
\bigr) + Q_0^{(x')} \bigr)
\\
&&\qquad= \bigl(I-Q_0^{(x)}\bigr) \otimes Q_0^{(x')}
+ Q_0^{(x)} \otimes\bigl(I-Q_0^{(x')}
\bigr) + 2 Q_0^{(x)} \otimes Q_0^{(x')}
,
\end{eqnarray*}
the projection onto $ (H_x \ominus H_0) \otimes H_0
\oplus H_0 \otimes(H_{x'} \ominus H_0) $ plus twice the projection
onto $ H_0 \otimes H_0 $ ($=H_0$). Thus, the equality $ f =
(Q_x + Q_{x'}) f $ [item (a)] becomes $ f \in(H_x \ominus H_0)
\otimes H_0 \oplus H_0 \otimes(H_{x'} \ominus H_0) $, or
equivalently, orthogonality of $ f $ to $ H_0 $ and $ (H_x
\ominus H_0) \otimes(H_{x'} \ominus H_0) $, which is item (b).
\end{pf}

%
%re5.7 #&#
\begin{remark}\label{5c2}
The proof given above shows also that
\[
\{ f \in H\dvtx f = Q_x f + Q_{x'} f \} = ( H_x
\ominus H_0 ) \oplus( H_{x'} \ominus H_0 )
\]
for all $ x \in B $.
\end{remark}

Let $ B_0 \subset B $ be a noise-type subalgebra, and $ f \in
H^{(1)}(B_0) $. We say that $ f $ is $B_0$-atomless, if for every $
\eps> 0 $ there exist $ n $ and $ x_1,\ldots,x_n \in B_0 $ such that
$ x_1 \vee\cdots\vee x_n = 1_\La$ and $ \| Q_{x_1} f \| \le\eps,
\ldots, \| Q_{x_n} f \| \le\eps$.

%
%pr5.8 #&#
\begin{proposition}\label{5c3}
If $ f \in H^{(1)}(B_0) $ is $B_0$-atomless, then $ f \in H^{(1)}(B)
$.
\end{proposition}

\begin{pf}
Given $ x \in B $, we have to prove that $ f = Q_x f + Q_{x'} f $. Let
$ g \in H_x \ominus H_0 $, $ h \in H_{x'} \ominus H_0 $; by Lemma
\ref{5c1} it is sufficient to prove that \mbox{$ \Ex( f g h ) = 0 $}.

Given $ \eps> 0 $, we take $ y_1,\ldots,y_n $ in $ B_0 $ such that $
y_1 \vee\cdots\vee y_n = 1_\La$, $ \| Q_{y_i} f \| \le\eps$ for
all $ i $, and in addition, $ y_i \wedge y_j = 0_\La$ whenever $ i
\ne j $. We have $ f = \sum_i Q_{y_i} f $ by Lemma \ref{5b1}, thus, $
\Ex( f g h ) = \sum_i \Ex( (Q_{y_i} f) g h ) $. Further, $ \Ex(
(Q_{y_i} f) g h ) = \langle Q_{y_i} f,g \otimes h \rangle= \langle
Q_{y_i} f,Q_{y_i} (g \otimes h) \rangle= \langle Q_{y_i} f
,(Q^{(x)}_{u_i} \otimes Q^{(x')}_{v_i}) (g \otimes h) \rangle=
\langle Q_{y_i} f,\break (Q^{(x)}_{u_i} g) \otimes(Q^{(x')}_{v_i} h)
\rangle$, where $ u_i = y_i \wedge x $ and $ v_i
= y_i \wedge x' $; it follows that\break  $ | \Ex( f g h ) | \le\sum_i \|
Q_{y_i} f \| \cdot\| Q^{(x)}_{u_i} g \| \cdot\| Q^{(x')}_{v_i} h \|
$. By \eqref{5a5}, $ \sum_i \| Q^{(x)}_{u_i} g \|^2 \le\| g \|^2 $
and $ \sum_i \| Q^{(x')}_{v_i} h \|^2 \le\| h \|^2 $. We get $ | \Ex
( f g h ) | \le( \max_i \| Q_{y_i} f \| ) ( \sum_i \|
Q^{(x)}_{u_i} g \| \cdot\break \| Q^{(x')}_{v_i}
h \| ) \le\eps\|g\| \|h\| $ for all $ \eps$.
\end{pf}

\begin{pf*}{Proof of Theorem \ref{1e2}}
Given an atomless noise-type subalgebra $ B_0 \subset B $, we have to
prove that $ H^{(1)}(B_0) \subset H^{(1)}(B) $. Applying Proposition
\ref{5a1} to $ B_0 $ we see that every $ f \in H^{(1)}(B_0) $ is
$B_0$-atomless. By Proposition \ref{5c3}, $ f \in H^{(1)}(B) $.
\end{pf*}

%s6 #&#
\section{\texorpdfstring{The easy part of Theorem \protect\ref{1c2}}
{The easy part of Theorem 1.5}}
\label{sec6}
% \input{sect6}

%s6.1 #&#
\subsection{From $\mathrm{(a)}$ to $\mathrm{(b)}$}
\label{6a}

In this subsection we assume that $ B $ is a \emph{classical}
noise-type Boolean algebra and prove that its completion, $ C $, is
equal to its closure, $ \Cl(B) $; in combination with Corollary \ref
{4b6} it
gives the implication (a) \imp(b) of Theorem \ref{1c2}.

The first chaos space $ H^{(1)} $ is invariant under $ Q_x $ for $ x
\in B $ and moreover, for $ x \in\Cl(B) $, as noted in
Section~\ref{5b}. We denote by $ \Down(x) $, for $ x \in\Cl(B) $, the
restriction of $ Q_x $ to $ H^{(1)} $ (treated as an operator $
H^{(1)} \to H^{(1)} $),\vadjust{\goodbreak} recall Section~\ref{5b} and note that
%
%
%e6.1 #&#
%e6.2 #&#
%e6.3 #&#
%e6.4 #&#
%
\begin{eqnarray}
&\Down(x)\dvtx H^{(1)} \to H^{(1)}, \qquad\Down(x) f = Q_x f
\qquad\mbox{for } x \in\Cl(B);& \label{6a1}
\\
&\Down(x) = Q_x^{(1)} \qquad\mbox{for } x \in B;& \label{6a2}
\\
&\Down(x) + \Down\bigl(x'\bigr) = I\qquad \mbox{for } x \in B;&
\label{6a3}
\\
&x \le y \mbox{ implies } \Down(x) \le\Down(y) \qquad\mbox{for } x,y \in
\Cl(B).&
\label{6a4}
\end{eqnarray}

We denote by $ \Q$ the closure of $ \{ \Down(x)\dvtx x \in B \} $ in the
strong operator topology; $ \Q$ is a closed set of commuting
projections on $ H^{(1)} $; we have $ \Down(x) \in\Q$ for $ x \in B
$, and by continuity for $ x \in\Cl(B) $ as well.

Note that $ q \in\Q$ implies $ I-q \in\Q$ [since $ \Down(x_n)
\to q $ implies $ \Down(x'_n) \to I-q $ by \eqref{6a3}].

For $ q \in\Q$ we define $ \Up(q) = \sigma(qH^{(1)}) \in\La$ (the
$\sigma$-field generated by $ q f $ for all $ f \in H^{(1)} $) and
note that
%
%
%e6.5 #&#
%e6.6 #&#
%
\begin{eqnarray}
&q_1 \le q_2 \qquad\mbox{implies } \Up(q_1) \le
\Up(q_2); &\label{6a5}
\\
&\Up(q) \vee\Up(I-q) = 1_\La\qquad\mbox{for } q \in\Q;&\label{6a6}
\end{eqnarray}
in general, $ \Up(q) \vee\Up(I-q) = \sigma(H^{(1)}) $, since $
qH^{(1)} + (I-q)H^{(1)} = H^{(1)} $; and the equality $
\sigma(H^{(1)}) = 1_\La$ is the classicality (Definition \ref{1a3}).

%
%le6.1 #&#
\begin{lemma}\label{6a7}
$ \Up(q) $ and $ \Up(I-q) $ are independent (for each $ q \in\Q$).
\end{lemma}

\begin{pf}
We take $ x_n \in B $ such that $ \Down(x_n) \to q $, then $
\Down(x'_n) \to I-q $. We have to prove that $ \sigma(qH^{(1)}) $ and $
\sigma((I-q)H^{(1)}) $ are independent, that is, two random vectors $ (
q f_1,\ldots,q f_k ) $ and $ ( (I-q)g_1,\ldots,(I-q)g_l ) $
are independent for all $ k,l $ and all $
f_1,\ldots,f_k,g_1,\ldots,g_l \in H^{(1)} $. It follows by
Fact \ref{2c10} from the similar claim for $ \Down(x_n) $ in place of
$ q
$.
\end{pf}

%
%le6.2 #&#
\begin{lemma}\label{6a8}
$ \Up(\Down(x)) = x $ for every $ x \in B $.
\end{lemma}

\begin{pf}
Denote $ q = \Down(x) $, then $ \Down(x') = I-q $ by \eqref{6a3}. We
have $ \Up(q) \le x $ (since $ qf = Q_x f $ is $x$-measurable for $ f
\in H^{(1)} $); similarly, $ \Up(I-q) \le x' $. By \eqref{6a6} and
\eqref{2c8}, $ \Up(q) = ( \Up(q) \vee\Up(I-q) ) \wedge x = x $.
\end{pf}

%
%le6.3 #&#
\begin{lemma}\label{6a9}
If $ q,q_1,q_2,\ldots\in\Q$ satisfy $ q_n \uparrow q $, then $
\Up(q_n) \uparrow\Up(q) $.
\end{lemma}

\begin{pf}
$ q_n H^{(1)} \uparrow qH^{(1)} $ implies $ \sigma(q_n H^{(1)})
\uparrow
\sigma(qH^{(1)}) $.
\end{pf}

%
%le6.4 #&#
\begin{lemma}\label{6a10}
If $ q,q_1,q_2,\ldots\in\Q$ satisfy $ q_n \downarrow q $, then $
\Up(q_n) \downarrow\Up(q) $.
\end{lemma}

\begin{pf}
We have $ \Up(q_n) \downarrow x $ for some $ x \in\La$, $ x \ge
\Up(q) $. By Lemma \ref{6a7}, $ \Up(q_n) $ and $ \Up(I-q_n) $ are
independent;
thus, $ x $ and $ \Up(I-q_n) $ are independent for all $ n $. By
Lemma \ref{6a9}, $ \Up(I-q_n) \uparrow\Up(I-q) $. Therefore $ x $
and $
\Up(I-q) $ are independent. By \eqref{6a6} and \eqref{2c8}, $ \Up
(q) =
( \Up(q) \vee\Up(I-q) ) \wedge x = x $.
\end{pf}

Now we prove that $ C = \Cl(B) $. By \eqref{4a3}, every $ x \in\Cl(B)
$ is of the form
\[
x = \liminf_n x_n = \sup_n
\inf_k x_{n+k}
\]
for some $ x_n \in B $. It follows that $ \Down(x) = \liminf_n
\Down(x_n) $; by Lem\-mas~\ref{6a9}, \ref{6a10}, $ \Up(\Down(x)) =
\liminf_n
\Up(\Down(x_n)) $; using Lemma \ref{6a8} we get $ \Up(\Down(x)) =
\liminf_n
x_n = x $.

On the other hand, $ I - \Down(x) = \limsup_n (I-\Down(x_n)) =\break
\limsup_n  \Down(x'_n) $ by \eqref{6a3}, thus the element $ y =
\Up(I-\Down(x)) $ satisfies (by Lemmas~\ref{6a9}, \ref{6a10} and
\ref{6a8}
again) $ y = \limsup_n \Up(\Down(x'_n)) = \limsup_n x'_n \in\Cl(B)
$.

By Lemma \ref{6a7}, $ x $ and $ y $ are independent. By \eqref{6a6},
$ x
\vee y = 1_\La$. Therefore, $ y $ is the complement of $ x $ in $ \Cl(B)
$, and we conclude that $ x \in C $. Thus, $ C = \Cl(B) $.

%s6.2 #&#
\subsection{From $\mathrm{(b)}$ to $\mathrm{(c)}$}
\label{6b}

As before, $ C $ stands for the completion of $ B $. Let $ x \in
\Cl(B) $ be such that $ x_n \uparrow x $ for some $ x_n \in B $.

%
%pr6.5 #&#
\begin{proposition}\label{6b1}
The following five conditions on $ x $ are equivalent:
\begin{longlist}[(a)]
\item[(a)] $ x \in C $;

\item[(b)] $ x \vee\lim_n x'_n = 1_\La$ for some $ x_n \in B $
satisfying $
x_n \uparrow x $;

\item[(c)] $ x \vee\lim_n x'_n = 1_\La$ for all $ x_n \in B $
satisfying $
x_n \uparrow x $;

\item[(d)] $ \lim_m \lim_n ( x_m \vee x'_n ) = 1_\La$ for some $
x_n \in B $
satisfying $ x_n \uparrow x $;

\item[(e)] $ \lim_m \lim_n ( x_m \vee x'_n ) = 1_\La$ for all $
x_n \in B $
satisfying $ x_n \uparrow x $.
\end{longlist}
\end{proposition}

%
%le6.6 #&#
\begin{lemma}\label{6b2}
$ (\sup_n x_n) \wedge(\inf_n x'_n) = 0_\La$ for every increasing
sequence $ (x_n)_n $ of elements of $ B $.
\end{lemma}

\begin{pf}
Note that $ x_m \wedge(\inf_n x'_n) \le x_m \wedge x'_m = 0_\La$, and
use \eqref{4a9}.
\end{pf}

\begin{pf*}{Proof of Proposition \ref{6b1}}
(c) \imp(b): trivial.

(b) \imp(a): by Lemma \ref{6b2}, $ x \wedge\lim_n x'_n = 0_\La$,
thus, $
x $ has the complement $ \lim_n x'_n $ and therefore belongs to $ C $.

(a) \imp(c): if $ x_n \uparrow x $, then (taking complements in the
Boolean algebra~$ C $) $ x'_n \ge x' $, therefore $ \lim_n x'_n \ge
x' $ and $ x \vee\lim_n x'_n \ge x \vee x' = 1_\La$.

We see that (a)\,\equ\,(b)\,\equ\,(c); Lemma \ref{6b3} below gives (b)\,\equ
(d) and (c)\,\equ\,(e).
\end{pf*}

%
%le6.7 #&#
\begin{lemma}\label{6b3}
For every increasing sequence $ (x_n)_n $ of elements of $ B $,
\[
\Bigl( \lim_n x_n \Bigr) \vee\Bigl( \lim
_n x'_n \Bigr) = \lim
_m \lim_n \bigl( x_m \vee
x'_n \bigr).
\]
\end{lemma}

\begin{pf}
Denote for convenience $ y = \lim_n x_n $ and $ z = \lim_n x'_n $. We
have $ x'_n \le x'_m $ for $ n \ge m $. Applying Theorem \ref{3d2} to
the pairs $ (x_m,x'_n) \in\La_{x_m} \times\La_{x'_m} $ for a fixed $
m $ and all $ n \ge m $ we get $ x_m \vee x'_n \to x_m \vee z $ as $ n
\to\infty$. Further, $ x_m \wedge z \le x_m \wedge x'_m = 0 $ for
all $ m $; by \eqref{4a9}, $ y \wedge z = 0 $, and by \eqref{4a8}, $ y
$ and $ z $ are independent. Applying Theorem \ref{3d2} (again) to $
(x_m,z) \in\La_y \times\La_z $ we get $ x_m \vee z \to y \vee z $ as
$ m \to\infty$. Finally, $ \lim_m \lim_n ( x_m \vee x'_n ) = \lim_m
( x_m \vee z ) = y \vee z = ( \lim_n x_n ) \vee( \lim_n x'_n ) $.
\end{pf}

By Corollary \ref{4b6}, condition (b) of Theorem \ref{1c2} is
equivalent to $ C
= \Cl(B) $. If it is satisfied, then Proposition \ref{6b1} gives $ (
\sup_n x_n )
\vee( \inf_n x'_n ) = 1_\La$ for all $ x_n \in B $ such that $ x_1
\le x_2 \le\cdots$, which is condition (c) of Theorem \ref{1c2}.

%s7 #&#
\section{\texorpdfstring{The difficult part of Theorem \protect\ref{1c2}}
{The difficult part of Theorem 1.5}}
\label{sec7}

The proof of the implication (c) \imp(a) of Theorem \ref{1c2}, given
in this section, is a remake of \cite{Ts03}, Sections~6c/6.3. In both
cases spectrum is crucial. The one-dimensional framework used in
\cite{Ts03} leads to ``spectral sets''---random compact subsets of
the parameter space $ \R$. The Boolean framework used here, being
free of any parameter space, leads to a more abstract ``spectral
space''; see Section~\ref{7b}. The number of points in a spectral set,
used in~\cite{Ts03}, becomes here a special function (denoted by $ K $
in Section~\ref{7d}) on the spectral space.

%s7.1 #&#
\subsection{A random supremum}
\label{7a}

By Proposition \ref{6b1}, condition (c) of Theorem \ref{1c2} may be
reformulated
as follows:
%
%
%e7.1 #&#
%
\begin{equation}
\label{7a1} \sup_n x_n \in C\qquad \mbox{for all }
x_n \in B \mbox{ such that } x_1 \le x_2 \le
\cdots
\end{equation}
or equivalently,
%
%
%e7.2 #&#
%
\begin{equation}
\label{7a2} \lim_m \lim_n \bigl(
x_1 \vee\cdots\vee x_m \vee( x_1 \vee\cdots
\vee x_n )' \bigr) = 1\qquad \mbox{for all } x_n
\in B.
\end{equation}

In order to effectively use this condition we choose a sequence $
(x_n)_n $, $ x_n \in B $, whose supremum is unlikely to belong to $ C
$. Ultimately it will be proved that $ \sup_n x_n \in C $ only if $ B
$ is classical.

However, we do not construct $ (x_n)_n $ explicitly. Instead we use
probabilistic method: construct a random sequence that has the needed
property with a nonzero probability.

Our noise-type Boolean algebra $ B $ consists of sub-$\sigma$-fields
on a
probability space $ (\Om,\F,P) $. However, randomness of $ x_n $ does
not mean that $ x_n $ is a function on $ \Om$. Another probability
space, unrelated to $ (\Om,\F,P) $, is involved. It may be thought of
as the space of sequences $ (x_n)_n $ endowed with a probability
measure described below.

A measure on a Boolean algebra $ b $ is defined as a countably
additive function $ b \to[0,\infty) $ (\cite{Ha}, Section~15). However,
the distribution of a random element of $ b $ (assuming that $ b $ is
finite) is rather a probability measure $ \nu$ on the set of all
elements of $ b $, that is, a countably additive function $ \nu\dvtx 2^b
\to
[0,\infty) $, $ \nu(b) = 1 $. It boils down to a function $ b \to
[0,\infty) $, $ x \mapsto\nu(\{x\}) $, such that $ \sum_{x\in
b} \nu(\{x\}) = 1 $.

Given a finite Boolean algebra $ b $ and a number $ p \in(0,1) $, we
introduce a probability measure $ \nu_{b,p} $ on the set of elements
of $ b $ by
%
%
%e7.3 #&#
%
\begin{equation}
\label{7a3}\quad  \nu_{b,p} \bigl( \{ a_{i_1} \vee\cdots\vee
a_{i_k} \} \bigr) = p^k (1-p)^{n-k}\qquad \mbox{for } 1
\le i_1 < \cdots< i_k \le n
\end{equation}
[using the notation of \eqref{2e1}]. That is, each atom is included
with probability $ p $, independently of others.

Given finite Boolean subalgebras $ b_1 \subset b_2 \subset\cdots
\subset B $ and numbers $ p_1,p_2,\ldots\in(0,1) $, we consider
probability measures $ \nu_n = \nu_{b_n,p_n} $ and their product, the
probability measure $ \nu= \nu_1 \times\nu_2 \times\cdots$ on the
set $ b_1 \times b_2 \times\cdots$ of sequences $ (x_n)_n $, $ x_n
\in b_n $. We note that $ \sup_n x_n \in\Cl(B) $ for all such
sequences and ask, whether or not
%
%
%e7.4 #&#
%
\begin{equation}
\label{7a4} \sup_n x_n \in C \qquad\mbox{for $
\nu$-almost all sequences } (x_n)_n,
\end{equation}
or equivalently,
%
%
%e7.5 #&#
%e7.6 #&#
%
\begin{eqnarray}
\label{7a5}&& \lim_m \lim_n \bigl(
x_1 \vee\cdots\vee x_m \vee( x_1 \vee\cdots
\vee x_n )' \bigr) = 1_\La
\nonumber
\\[-8pt]
\\[-8pt]
\eqntext{\mbox{for $\nu$-almost all sequences } (x_n)_n.}
\end{eqnarray}

%
%pr7.1 #&#
\begin{proposition}\label{7a6}
If \eqref{7a5} holds for all such $ b_1,b_2,\ldots$ and $
p_1,p_2,\ldots, $ then~$ B $ is classical.
\end{proposition}

In order to prove the implication (c) \imp(a) of Theorem \ref{1c2} it
is sufficient to prove Proposition \ref{7a6}. To this end we need
spectral theory.

%s7.2 #&#
\subsection{Spectrum as a measure class factorization}
\label{7b}

The projections $ Q_x $ for $ x \in\Cl(B) $ commute by \eqref{4a5},
and generate a commutative von Neumann algebra $ \A$. Section~\ref{2d}
gives us a measure class space $ (S,\Si,\M) $ and an isomorphism
%
%
%e7.7 #&#
%
\begin{equation}
\label{7b1} \al\dvtx \A\to L_\infty(S,\Si,\M).
\end{equation}
We call $ (S,\Si,\M) $ (endowed with $ \al$) the \emph{spectral
space} of $ B $. Projections $ Q_x $ turn into indicators
%
%
%e7.8 #&#
%
\begin{equation}
\label{7b2} \al(Q_x) = \One_{S_x}, \qquad S_x \in
\Si\qquad \mbox{ for } x \in\Cl(B)
\end{equation}
(of course, $ S_x $ is an equivalence class rather than a set);
\eqref{4a7} gives
%
%
%e7.9 #&#
%
\begin{equation}
\label{7b3} S_x \cap S_y = S_{x \wedge y} \qquad\mbox{for
} x,y \in\Cl(B).
\end{equation}
(In contrast, the evident inclusion $ S_x \cup S_y \subset S_{x\vee y}
$ is generally strict.)

\begin{claim*}
%
%
%e7.10 #&#
%
\begin{equation}
\label{7b35} x_n \downarrow x \mbox{ implies } S_{x_n}
\downarrow S_x;\qquad \mbox{ also } x_n \uparrow x \mbox{
implies } S_{x_n} \uparrow S_x;
\end{equation}
here $ x,x_1,x_2,\ldots\in\Cl(B) $.
\end{claim*}

\begin{pf} let $ x_n
\uparrow x $, then $ Q_{x_n} \uparrow Q_x $, thus $ \al(Q_{x_n})
\uparrow\al(Q_x) $ by \eqref{2d11}, which means $ S_{x_n}
\uparrow S_x $; the case $ x_n \downarrow x $ is similar.
\end{pf}

The subspaces $ H_x = Q_x H \subset H $ for $ x \in\Cl(B) $ are a
special case of the subspaces $ H(E) = \al^{-1}(\One_E) H \subset H $
for $ E \in\Si$ [recall \eqref{2d9}]; by \eqref{7b2},
%
%
%e7.11 #&#
%
\begin{equation}
\label{7b4} H (S_x) = H_x\qquad \mbox{for } x \in\Cl(B).
\end{equation}

Every subset of $ B $ leads to a subalgebra of $ \A$. In particular,
for every $ x \in B $ we introduce the von Neumann algebra
%
%
%e7.12 #&#
%
\begin{eqnarray}
\label{7b5} \A_x \subset\A
\nonumber
\\[-8pt]
\\[-8pt]
\eqntext{\mbox{generated by } \{ Q_y\dvtx y \in B,
x \vee y = 1_\La\} = \{ Q_{u\vee x'}\dvtx u \in B, u \le x\}}
\end{eqnarray}
and the $\sigma$-field $ \Si_x \subset\Si$ such that
%
%
%e7.13 #&#
%
\begin{equation}
\label{7b6} \al(\A_x) = L_\infty(\Si_x)\qquad
\mbox{for } x \in B
\end{equation}
(see Fact \ref{2d3}). Note that
%
%
%e7.14 #&#
%
\begin{equation}
\label{7b7} x \le y\qquad \mbox{implies } \A_x \subset\A_y
\mbox{ and } \Si_x \subset\Si_y\qquad \mbox{for } x,y \in B.
\end{equation}

Recall Notation \ref{2c2}: $ Q_u^{(x)}\dvtx H_x \to H_x $ for $ u \le x
$, and
Fact \ref{2c3}: given independent $ x,y $, treating $ H_{x\vee y} $ as $
H_x \otimes H_y $ we have $ Q_{u\vee v} = Q_u^{(x)} \otimes Q_v^{(y)}
$ for all $ u \le x $, $ v \le y $. Introducing von Neumann algebras $
\A^{(x)} $ of operators on $ H_x $,
%
%
%e7.15 #&#
%
\begin{equation}
\label{7b75} \A^{(x)} \mbox{ generated by } \bigl\{ Q_u^{(x)}
\dvtx u \in B, u \le x \bigr\},
\end{equation}
we get
%
%
%e7.16 #&#
%
\begin{equation}
% \label{7b7}
\A^{(x\vee y)} = \A^{(x)} \otimes\A^{(y)}\qquad
\mbox{whenever } x \wedge y = 0, x,y \in B.
\end{equation}
In the case $ y = x' $, treating $ H $ as $ H_x \otimes H_{x'} $ we
have
%
%
%e7.17 #&#
%
\begin{equation}
\label{7b9} \A= \A^{(x)} \otimes\A^{(x')}\quad \mbox{and}\quad
\A_x = \A^{(x)} \otimes I\qquad \mbox{for } x \in B
\end{equation}
(for the latter, fix $ v = x' $),---a natural isomorphism between $
\A_x $ and $ \A^{(x)} $. Thus, $ \al(\A^{(x)} \otimes I) =
L_\infty(\Si_x) $, $ \al(I \otimes\A^{(x')}) = L_\infty(\Si
_{x'}) $
and $ \al(\A^{(x)} \otimes\A^{(x')}) = L_\infty(\Si) $. By
Fact \ref{2d8}, for all $ x \in B $,
%
%
%e7.18 #&#
%e7.19 #&#
%
\begin{eqnarray}
&\Si_x \mbox{ and } \Si_{x'} \mbox{ are $\M$-independent},&
\\
&\Si_x \vee\Si_{x'} = \Si.&
\end{eqnarray}

(Thus, the spectral space is a measure class (or ``type'')
factorization as defined in \cite{TV}, Section~1c and discussed
in \cite{Ar}, Section~14.4, \cite{Ts04}, Section~10.)

%
%re7.2 #&#
\begin{remark}\label{7b10}
The closure of $ B $ determines uniquely the algebra $ \A$ and
therefore also the spectral space.
\end{remark}

%
%ex7.3 #&#
\begin{example}
Let a noise-type Boolean algebra $B$ be finite, with $n$ atoms. Then
$\A$ is of dimension $2^n$; $(S,\Si,\M)$ is the discrete space with
$2^n$ points. Up to isomorphism we may treat both $B$ and $S$ as
consisting of all subsets of $\Atoms(B)$, and then $S_x$ consists of all
subsets of $x$.
\end{example}

%
%ex7.4 #&#
\begin{example}
Let $B$ and $y_n$ be as in Section~\ref{1b} and Example \ref{1d35}. The
sign change transformation $ \Om\to\Om$ decomposes the Hilbert
space: $ H = H_{\even} \oplus H_{\odd} $. Introducing $ y = \sup_n
y_n \in\Cl(B) \setminus B $ we have $ H_y = H_{\even} $; the projection
$ Q_y $ onto $ H_{\even} $ corresponds to the indicator of $ S_y $. Up
to isomorphism we may treat $ S_y $ as consisting of all finite
subsets of $ \{1,2,\ldots\} $, and $ S \setminus S_y $ as consisting of
their complements, the cofinite subsets of $ \{1,2,\ldots\} $. Both $B$
and $S$ become the same countable set, and $S_x$ consists of all
finite/cofinite subsets of $x$ (i.e., finite subsets of a finite
$x$, but finite/cofinite subsets of a cofinite $x$). See
also \cite{Ts04}, Section~9a (for $m=2$).
\end{example}

%
%ex7.5 #&#
\begin{example}
Let $B$ correspond to a noise over $\R$ (see Section~\ref{1f}), and
assume that the noise is classical, which is equivalent to
classicality of $B$ (as defined by Definition \ref{1a3}); it is also
equivalent to existence of L\'evy processes whose increments generate
the noise. Assume that the noise is not trivial, that is, $
1_B \ne0_B $. Then $B$ as a Boolean algebra is isomorphic to the
Boolean algebra of all finite unions of intervals (on $\R$) modulo
finite sets. Up to isomorphism we may treat $ (S,\Si,\M) $ as the
space of all finite subsets of $\R$; a measure $\mu$ on $S$ belongs to
$\M$ if and only if $\mu$ is equivalent (i.e., mutually absolutely continuous)
to the (symmetrized) $n$-dimensional Lebesgue measure on the subset $
S_n \subset S $ of all $n$-point sets, for every $ n = 0,1,2,\ldots$;
for $n=0$ it means an atom: $ \mu(\{\varnothing\}) > 0 $. As before, $
S_x $ consists of all $ s \in S $ such that $ s \subset x $; but now
$S$ and $B$ are quite different collections of sets. See
also \cite{Ts04}, Example 9b9.

In contrast, for a black noise the elements of $S$ may be thought of
as some perfect compact subsets of $\R$ (including the empty set), of
Lebesgue measure zero. And if a noise is neither classical nor black,
then all finite sets belong to $S$, but also some infinite compact
sets of Lebesgue measure zero belong to $S$. These may be countable or
not, depending on the noise. See also \cite{Ts04}, Sections~9b, 9c.
\end{example}

%s7.3 #&#
\subsection{\texorpdfstring{Restriction to a sub-$\sigma$-field}
{Restriction to a sub-sigma-field}}
\label{7c}

As was noted in Section~\ref{3d}, for an arbitrary $ x \in\La$ the
triple $ (\Om,x,P|_x) $ is also a probability space, and its lattice
of $\sigma$-fields is naturally embedded into $ \La$,
\[
\La(\Om,x,P|_x) = \La_x = \{ y \in\La\dvtx y \le x \}
\subset\La.
\]
Dealing with a noise-type Boolean algebra $ B \subset\La$ over $
(\Om,\F,P) $, we introduce
\[
B_x = B \cap\La_x = \{ u \in B\dvtx u \le x \} \subset B\qquad
\mbox{for } x \in B
\]
and note that
\[
B_x \subset\La_x \mbox{ is a noise-type Boolean
algebra over } (\Om,x,P|_x);
\]
thus, notions introduced for $ B $ have their counterparts for $ B_x
$. We mark them by the \emph{left} index $ x $. Some of these
counterparts were used in previous (sub)sections. For $ x \in B $:
\begin{eqnarray*}
& _x H = H_x; \qquad\mbox{see Section~\ref{2c}},&
\\
& _x Q_u = Q_u^{(x)}\qquad \mbox{for
} u \in B_x; \qquad\mbox{see Notation \ref{2c2}},&
\\
&{_x} \A= \A^{(x)};\qquad \mbox{see \eqref{7b75}},&
\\
& {_x} S = S; \leftidx{_x} \Si= \Si_x;
\qquad\mbox{see \eqref{7b6}},&
\\
& {_x} \al\dvtx \A^{(x)} \to L_\infty(
\Si_x), \qquad {_x} \al(A) = \al(A\otimes I);\qquad \mbox{see
\eqref{7b9}},&
\\
& {_x} S_u = S_{u\vee x'} \qquad\mbox{for } u \in
B_x; \qquad\mbox{see \eqref{7b2}},&
\\
&{_x} H^{(1)} = H^{(1)} \cap H_x;&
\end{eqnarray*}
the last line follows easily from Lemma \ref{5b1}; the next to the
last line
holds, since $ \leftidx{_x} \al( \leftidx{_x} Q_u ) = \al( Q_u^{(x)}
\otimes I ) = \al( Q_u^{(x)} \otimes Q_{x'}^{(x')} ) = \al( Q_{u\vee
x'} ) $. The counterpart of $ H(E) = \al^{-1}(\One_E)H $ for $ E \in
\Si$ is $ \leftidx{_x} H(E) = \leftidx{_x} \al^{-1}(\One_E) H_x $ for
$ E \in\Si_x $.

%
%le7.6 #&#
\begin{lemma}\label{7c1}
For every $ x \in B $, treating $ H $ as $ H_x \otimes H_{x'} $ we
have $ H ( E \cap F ) = ( \leftidx{_x} H(E) ) \otimes(
\leftidx{_{x'}} H(F) ) $ for all $ E \in\Si_x $, $ F \in\Si_{x'} $.
\end{lemma}

\begin{pf}
We take $ A \in\A^{(x)} $, $ B \in\A^{(x')} $ such that $ \al( A
\otimes I ) = \One_E $, $ \al( I \otimes B ) = \One_F $, then $ \al(
A \otimes B ) = \One_E \One_F = \One_{E\cap F} $ and $ H ( E \cap F )
= ( A \otimes B ) ( H_x \otimes H_{x'} ) = ( A H_x ) \otimes( B
H_{x'} ) = ( \leftidx{_x} H(E) ) \otimes( \leftidx{_{x'}} H(F) ) $.
\end{pf}

%s7.4 #&#
\subsection{Classicality via spectrum}
\label{7d}

Let $ b \subset B $ be a finite Boolean subalgebra. For almost every $
s \in
S $ the set $ \{ x \in b\dvtx s \in S_x \} $ is a filter on $ b $ due to
\eqref{7b3}; like every filter on a finite Boolean algebra, it is
generated by some $ x_b(s) \in b $,
%
%
%e7.20 #&#
%
\begin{equation}
\label{7d1} \forall x \in b \qquad\bigl( s \in S_x \equiv x \ge
x_b(s) \bigr).
\end{equation}
Like every element of $ b $, $ x_b(s) $ is the union of some of the
atoms of $ b $ [recall \eqref{2e2}]; the number of these atoms will be
denoted by $ K_b(s) $,
\[
K_b(s) = \bigl| \bigl\{ a \in\Atoms(b)\dvtx a \le x_b(s) \bigr\}
\bigr|.
\]
For two finite Boolean subalgebras,
%
%
%e7.21 #&#
%
\begin{equation}
\label{7d15} \mbox{if } b_1 \subset b_2 \mbox{ then }
K_{b_1}(\cdot) \le K_{b_2}(\cdot) \mbox{ and }
x_{b_1}(s) \ge x_{b_2}(s).
\end{equation}
Each $ K_b $ is an equivalence class (rather than a function), and the
set of all $ b $ need not be countable. We take supremum in the
complete lattice of all equivalence classes of measurable functions $
S \to[0,+\infty] $ (recall Section~\ref{2f}):
%
%
%e7.22 #&#
%
\begin{equation}
\label{7d17} K = \sup_b K_b, \qquad K\dvtx S \to[0,+
\infty],
\end{equation}
where $ b $ runs over all finite Boolean subalgebras $ b \subset B $.

%
%th7.7 #&#
\begin{theorem}\label{7d2}
$ B $ is classical if and only if $ K(\cdot) < \infty$ almost everywhere.
\end{theorem}

We split this theorem in two propositions as follows. Recall that
classicality is defined by Definition \ref{1a3} as the equality $
\sigma(H^{(1)}) =
1_\La$. Introducing
\[
E_k = \bigl\{ s \in S\dvtx K(s) = k \bigr\}\quad \mbox{and}\quad H^{(k)}
= H(E_k)\qquad \mbox{for } k=0,1,2,\ldots
\]
[recall \eqref{2d9}] we reformulate the condition $ K(\cdot) < \infty
$ as $ S = \biguplus_k E_k $ and further, by \eqref{2d10}, as $ H =
\bigoplus_k H^{(k)} $. For $ k=1 $ the new notation conforms to the old
one in the following sense.

%
%pr7.8 #&#
\begin{proposition}\label{7d3}
$ H(E_1) $ is equal to the first chaos space $ H^{(1)} $ (defined by
Definition \ref{1a2}).
\end{proposition}

%
%pr7.9 #&#
\begin{proposition}\label{7d4}
$ \sigma(H^{(k)}) \subset\sigma(H^{(1)}) $ for all $ k=2,3,\ldots.$
\end{proposition}

Thus, $ \bigoplus_k H^{(k)} = H \equiv\sigma(H^{(1)})
= \sigma(H) \equiv\sigma(H^{(1)}) = 1_\La$. We see that Theorem
\ref{7d2}
follows from Propositions \ref{7d3}, \ref{7d4}.

The proof of Proposition \ref{7d3} is given after three lemmas.

We introduce minimal nontrivial finite Boolean subalgebras $ b_x = \{
0, x, x', 1 \} $ for $ x \in B $.

%
%le7.10 #&#
\begin{lemma}\label{7d5}
For every $ x \in B $,
\[
\{ f \in H\dvtx f = Q_x f + Q_{x'} f \} = H \bigl( \bigl\{ s
\dvtx K_{b_x}(s) = 1 \bigr\} \bigr).
\]
\end{lemma}

\begin{pf}
$ \{ s\dvtx K_{b_x}(s) = 1 \} = \{ s\dvtx K_{b_x}(s) \le1 \}
\setminus\{ s\dvtx K_{b_x}(s) = 0 \} = ( S_x \cup S_{x'} ) \setminus S_0
= ( S_x \setminus S_0 ) \uplus( S_{x'} \setminus S_0 ) $ (since $ S_x
\cap S_{x'} = S_0 $), thus $ H ( \{ s\dvtx K_{b_x}(s) = 1 \} ) = H (
S_x \setminus S_0 ) \oplus H ( S_{x'} \setminus S_0 ) = ( H_x \ominus
H_0 ) \oplus( H_{x'} \ominus H_0 ) $; use Remark \ref{5c2}.
\end{pf}

%
%le7.11 #&#
\begin{lemma}\label{7d6}
Assume that $ b_1, b_2 \subset B $ are finite Boolean subalgebras, and
$ b \subset B $ is the (finite by Fact \ref{2e3}) Boolean subalgebra
generated by $ b_1, b_2 $. Then
\[
\bigl\{ s\dvtx K_{b_1}(s) \le1 \bigr\} \cap\bigl\{ s\dvtx
K_{b_2}(s) \le1 \bigr\} \subset\bigl\{ s\dvtx K_b(s) \le1
\bigr\}.
\]
\end{lemma}

\begin{pf}
If $ K_{b_1}(s) \le1 $, $ K_{b_2}(s) \le1 $ and $ s \notin S_0 $,
then $ x_{b_1}(s) \in\Atoms(b_1) $, $ x_{b_2}(s) \in\Atoms(b_2) $,
thus $ x_b(s) \le x_{b_1}(s) \wedge x_{b_2}(s) \in\Atoms(b) $ by
Fact \ref{2e3}, therefore $ K_b(s) \le1 $.
\end{pf}

%
%le7.12 #&#
\begin{lemma}\label{7d7}
$ \{ s\dvtx K(s) \le1 \} = \inf_{x\in B} \{ s\dvtx K_{b_x}(s) \le1
\} $,
and $ \{ s\dvtx K(s) = 1 \} = \inf_{x\in B} \{ s\dvtx K_{b_x}(s) = 1
\} $
(the infimum of equivalence classes).\vadjust{\goodbreak}
\end{lemma}

\begin{pf}
Every finite Boolean subalgebra $ b $ is generated by the Boolean
subalgebras $ b_x $ for $ x \in b $; by Lemma \ref{7d6}, $ \{ s\dvtx
K_b(s) \le
1 \} \supset\bigcap_{x\in b} \{ s\dvtx K_{b_x}(s) \le1 \} $; the infimum
over all $ b $ gives $ \{ s\dvtx K(s) \le1 \} \supset\inf_{x\in B}
\{ s
\dvtx K_{b_x}(s) \le1 \} $. The converse inclusion being trivial, we get
the first equality. The second equality follows, since the set $ \{ s
\dvtx K_b(s) = 0 \} $ is equal to $ S_0 $, irrespective of $ b $.
\end{pf}

\begin{pf*}{Proof of Proposition \ref{7d3}}
It follows from the second equality of Lem\-ma~\ref{7d7}, using \eqref{2f05},
that $ H(E_1) = \bigcap_{x\in B} H ( \{ s\dvtx K_{b_x}(s) = 1 \} )
$. Using Lem\-ma~\ref{7d5} we get $ H(E_1) = \bigcap_{x\in B} \{ f \in
H\dvtx f =
Q_x f + Q_{x'} f \} = H^{(1)} $.
\end{pf*}

In order to prove Theorem \ref{7d2} it remains to prove Proposition
\ref{7d4}.

We have $ K $ introduced for $ B $ by \eqref{7d17}, but also for $ B_x
$ we have $ \leftidx{_x} K $, the counterpart of $ K $ in the sense of
Section~\ref{7c};
\[
\leftidx{_x} K = \sup_b \leftidx{_x}
K_b,\qquad \leftidx{_x} K\dvtx S \to[0,\infty] \qquad\mbox{for } x \in
B,
\]
where $ b $ runs over all finite Boolean subalgebras $ b \subset B_x
$; $ \leftidx{_x} K $ is an equivalence class of $\Si_x$-measurable
functions $ S \to[0,\infty] $.

%
%le7.13 #&#
\begin{lemma}\label{7d10}
$ \leftidx{_{x\vee y}} K = \leftidx{_x} K + \leftidx{_y} K $ for all
$ x,y \in B $ such that $ x \wedge y = 0_\La$.
\end{lemma}

\begin{pf}
When calculating $ \leftidx{_{x\vee y}} K $ we may restrict
ourselves to finite subalgebras $ b \subset B_{x\vee y} $ that contain
$ x $ and $ y $; recall \eqref{7d15}. Each such $ b $ may be thought
of as a pair of $ b_1 \subset B_x $ and $ b_2 \subset B_y $. We have
$ \Atoms(b) = \Atoms(b_1) \uplus\Atoms(b_2) $, $ x_b(s) =
x_{b_1}(s) \vee x_{b_2}(s) $ (recall that $ \leftidx{_x} S_u =
S_{u\vee x'} $ for $ u \le x $), thus $ \leftidx{_{x\vee y}} K_b =
\leftidx{_x} K_{b_1} + \leftidx{_y} K_{b_2} $; take the supremum in $
b_1, b_2 $.
\end{pf}

%
%le7.14 #&#
\begin{lemma}\label{7d11}
$ \{ s \in S\dvtx K(s) = 2 \} = \sup_{x\in B} \{ s \in S\dvtx\leftidx{_x}
K(s) = \leftidx{_{x'}} K(s) = 1 \} $ (the supremum of equivalence
classes).
\end{lemma}

\begin{pf}
The ``$\supset$'' inclusion follows from Lemma \ref{7d10}; it is
sufficient to prove that $ \{ s \in S\dvtx K(s) = 2 \} \subset\bigcup
_{x\in
b_1\cup b_2\cup\cdots} \{ s \in S\dvtx\leftidx{_x} K(s) = \leftidx{_{x'}}
K(s) = 1 \} $ if $ b_1 \subset b_2 \subset\cdots$ satisfy $ K_{b_n}
\uparrow K $.

Given $ s $ such that $ K(s)=2 $, we take $ n $ such that $ K_{b_n}(s)
= 2 $, that is, $ x_{b_n}(s) $ contains exactly two atoms of $ b_n
$. We choose $ x \in b_n $ that contains exactly one of these two
atoms; then $ \leftidx{_x} K_{b_n}(s) = \leftidx{_{x'}} K_{b_n}(s) = 1
$, therefore $ \leftidx{_x} K(s) = \leftidx{_{x'}} K(s) = 1 $, since $
1 = \leftidx{_x} K_{b_n}(s) \le\leftidx{_x} K(s) = K(s) -
\leftidx{_{x'}} K(s) \le2 - \leftidx{_{x'}} K_{b_n}(s) = 1 $.
\end{pf}

We use the counterpart (in the sense of Section~\ref{7c}) of
Proposition \ref{7d3}: $ \leftidx{_x} H ( \leftidx{_x} E_1 ) =
\leftidx{_x} H^{(1)} $, that is, for every $ x \in B $,
%
%
%e7.23 #&#
%
\begin{equation}
\label{7d12} \leftidx{_x} H \bigl( \bigl\{ s \in S\dvtx\leftidx
{_x} K(s) = 1 \bigr\} \bigr) = H^{(1)} \cap
H_x.
\end{equation}

\begin{pf*}{Proof of Proposition \ref{7d4} for $ k=2 $}
It follows from Lemma \ref{7d11} and~\eqref{2f2} that $ H^{(2)} $ is
generated (as a closed linear subspace of $ H $) by the union, over
all $ x \in B $, of the subspaces $ H ( \{ s \in S\dvtx\leftidx{_x}
K(s) = \leftidx{_{x'}} K(s) = 1 \} ) $. In order to get $
\sigma(H^{(2)}) \subset\sigma(H^{(1)}) $ it is sufficient to prove that
%
%
%e7.24 #&#
%
\begin{equation}
\label{7d13} \sigma\bigl( H \bigl( \bigl\{ s \in S\dvtx\leftidx{_x}
K(s) = \leftidx{_{x'}} K(s) = 1 \bigr\} \bigr) \bigr) \subset
\sigma
\bigl(H^{(1)}\bigr)\qquad \mbox{for all } x \in B.
\end{equation}
By Lemma \ref{7c1} and \eqref{7d12}, $ H ( \{ s \in S\dvtx\leftidx
{_x} K(s)
= \leftidx{_{x'}} K(s) = 1 \} ) = \leftidx{_x} H ( \{ s \in S\dvtx\break
\leftidx{_x} K(s) = 1 \} ) \otimes\leftidx{_{x'}} H ( \{ s \in
S\dvtx\leftidx{_{x'}} K(s) = 1 \} ) = ( H_x \cap H^{(1)} ) \otimes( H_{x'}
\cap H^{(1)} ) $, which implies \eqref{7d13}.
\end{pf*}

The proof of Proposition \ref{7d4} for higher $ k $ is similar. Lemma
\ref{7d11} is generalized to
\[
\bigl\{ s \in S\dvtx K(s) = k \bigr\} = \sup_{x\in B} \bigl\{ s \in
S\dvtx\leftidx{_x} K(s) = k-1, \leftidx{_{x'}} K(s) = 1
\bigr\},
\]
and \eqref{7d13} to
\[
\sigma\bigl( H \bigl( \bigl\{ s \in S\dvtx\leftidx{_x} K(s) = k-1,
\leftidx{_{x'}} K(s) = 1 \bigr\} \bigr) \bigr) \subset\sigma
\bigl(H^{(k-1)} \cup H^{(1)}\bigr).
\]
Thus, $ \sigma(H^{(k)}) \subset\sigma(H^{(k-1)} \cup H^{(1)}) $. By
induction in $ k $, $ \sigma(H^{(k)}) \subset\sigma(H^{(1)}) $, which
completes the proof of Proposition \ref{7d4} and Theorem \ref{7d2}.

%s7.5 #&#
\subsection{Finishing the proof}
\label{7e}

%
%pr7.15 #&#
\begin{proposition}\label{7e1}
If \eqref{7a5} holds for all $ b_1 \subset b_2 \subset\cdots$ and $
p_1,p_2,\ldots\in(0,1) $, then $ K(\cdot) < \infty$ almost
everywhere.
\end{proposition}

By Theorem \ref{7d2}, in order to prove Proposition \ref{7a6} it is
sufficient to prove Proposition \ref{7e1}.

The relation $ \lim_m \lim_n ( y_m \vee y'_n ) = 1_\La$ for $ y_1
\le
y_2 \le\cdots$ [appearing in \eqref{7a5} with $ y_n = x_1 \vee\cdots
\vee x_n $] may be reformulated in spectral terms using \eqref{7b35};
it turns into $ \bigcup_m \bigcap_n S_{y_m \vee y'_n} = S $, in other words,
almost every $ s \in S $ satisfies $ \exists m \forall n \> s \in
S_{y_m \vee y'_n} $. Accordingly, \eqref{7a5} may be rewritten as
follows:
%
%
%e7.25 #&#
%e7.26 #&#
%
\begin{eqnarray}
\label{7e15} &&\mbox{for $\nu$-almost all sequences } (x_n)_n
, \mbox{for almost all } s \in S,\exists m \forall n
\nonumber
\\[-8pt]
\\[-8pt]
&&\eqntext{s \in S_{ x_1 \vee\cdots\vee x_m \vee(
x_1 \vee\cdots\vee x_n )' }.}
\end{eqnarray}

We choose $ p_1,p_2,\ldots\in(0,1) $ and $ c_1,c_2,\ldots\in
\{1,2,3,\ldots\} $ such that
%
%
%e7.27 #&#
%e7.28 #&#
%
\begin{eqnarray}
&\displaystyle\sum_n p_n < 1,& \label{7e2}
\\
&(1-p_n)^{c_n} \to0\qquad \mbox{as } n \to\infty.& \label{7e3}
\end{eqnarray}
We also choose finite Boolean subalgebras $ b_1 \subset b_2 \subset
\cdots\subset B $ such that $ K_{b_n} \uparrow K $ and introduce $ b =
b_1 \cup b_2 \cup\cdots\subset B $ (a countable Boolean
subalgebra).

\begin{claim*}
%
%e7.29 #&#
%
\begin{equation}
\leftidx{_x} K_{b_n} \uparrow\leftidx{_x} K\qquad
\mbox{for every } x \in b.
\end{equation}
\end{claim*}
\begin{pf}$ \leftidx{_x} K \ge\lim_n \leftidx{_x} K_{b_n} =
\lim_n (
K_{b_n} - \leftidx{_{x'}} K_{b_n} ) \ge K - \leftidx{_{x'}} K =
\leftidx{_x} K $.
\end{pf}

\begin{remark*} For $ x \in b_n $, by $ \leftidx{_x} K_{b_n} $ we mean
$ \leftidx{_x} K_{b_n \cap B_x } $. Thus $ \leftidx{_x} K_{b_n} $
is well defined for all $ n $ large enough, provided that $ x \in b
$.
\end{remark*}

Using Fact \ref{2f3} we take $ n_1 < n_2 < \cdots$ such that for almost
every $ s \in S $
%
%
%e7.30 #&#
%e7.31 #&#
%
\begin{eqnarray}
\label{7e4} &\mbox{either } \leftidx{_x} K(s) < \infty,&
\nonumber
\\[-8pt]
\\[-8pt]
\nonumber
&\mbox{or } \leftidx{_x} K_{b_{n_k}}(s) \ge c_k
\mbox{ for all $ k $ large enough.}&
\end{eqnarray}
These $ n_k $ depend on $ x \in b $. However, countably many $ x $ can
be served by a single sequence $ (n_k)_k $ using the well-known
diagonal argument. This way we ensure~\eqref{7e4} with a single $
(n_k)_k $ for all $ x \in b $. Now we rename $ b_{n_k} $ into $ b_k $,
discard a null set of bad points $ s \in S $ and get
%
%
%e7.32 #&#
%e7.33 #&#
%
\begin{eqnarray}
\label{7e5} &\mbox{either } \leftidx{_x} K(s) < \infty,&
\nonumber
\\[-8pt]
\\[-8pt]
\nonumber
&\mbox{or } \leftidx{_x} K_{b_n}(s) \ge c_n
\mbox{ for all $ n $ large enough}&
\end{eqnarray}
for all $ x \in b $ and $ s \in S $; here ``$n$ large enough'' means $
n \ge n_0(x,s) $.

We recall the product measure $ \nu= \nu_1 \times\nu_2 \times\cdots
$ introduced in Section~\ref{7a} on the product set $ b_1 \times b_2
\times\cdots$; as before, $ \nu_n = \nu_{b_n,p_n} $.
For notational convenience we treat the coordinate maps $ X_n\dvtx( b_1
\times b_2 \times\cdots, \nu) \to b_n $, $ X_n(x_1,\break x_2,\ldots) = x_n
$, as independent $b_n$-valued random variables; $ X_n $ is
distributed $ \nu_n $, that is, $ \mathbb{P}(X_n = x ) = \nu_n(\{x\}
) $ for
$ x \in b_n $. We introduce $b_n$-valued random variables
\[
Y_n = X_1 \vee\cdots\vee X_n.
\]

%
%le7.16 #&#
\begin{lemma}\label{7e6}
$ \mathbb{P}(\leftidx{_{Y'_n}} K(s) < \infty) \le p_1 + \cdots+ p_n
$ for
all $ s \in S $ such that $ K(s) = \infty$ and all $ n $.
\end{lemma}

\begin{pf}
There exists $ a \in\Atoms(b_n) $ such that $ \leftidx{_a} K(s) =
\infty$ [since\break  $ \sum_a \leftidx{_a} K(s) = K(s) = \infty$]. We have
$ \leftidx{_{Y'_n}} K(s) < \infty\impl a \le Y_n \impl\exists
k\in\{1,\ldots,\break n\} \> a \le X_k $, therefore $ \mathbb{P}(\leftidx
{_{Y'_n}} K(s) < \infty) \le\sum_{k=1}^n \mathbb{P}(a \le X_k ) =
\sum_{k=1}^n p_k $.
\end{pf}

%
%le7.17 #&#
\begin{lemma}\label{7e7}
If $ x \in b_m $ and $ s \in S $ satisfy $ \leftidx{_x} K(s)=\infty$,
then
\[
\mathbb{P}\bigl(\forall n > m \ X_n \wedge x \wedge
x_{b_n}(s) = 0_\La\bigr) = 0.
\]
\end{lemma}

\begin{pf}
For $ n > m $,
\[
\mathbb{P}\bigl(X_n \wedge x \wedge x_{b_n}(s) =
0_\La\bigr) = (1-p_n)^{\leftidx{_x}
K_{b_n}(s)},\vadjust{\goodbreak}
\]
since $ x \wedge x_{b_n}(s) $ contains $ \leftidx{_x} K_{b_n}(s) $
atoms of $ b_n $. By \eqref{7e5}, $ \leftidx{_x} K_{b_n}(s) \ge c_n $
for all $ n $ large enough. Thus, $ \mathbb{P}(X_n \wedge x \wedge
x_{b_n}(s) = 0_\La) \le(1-p_n)^{c_n} \to0 $ as $ n \to\infty$ by~\eqref{7e3}.
\end{pf}

%
%le7.18 #&#
\begin{lemma}\label{7e8}
\[
\mathbb{P}\bigl(\leftidx{_{Y'_m}} K(s)=\infty\mbox{ \small and }
\forall n > m \> s \in S_{Y_m\vee Y'_n} \bigr) = 0
\]
for all $ s \in S $ and $ m $.
\end{lemma}

\begin{pf}
By \eqref{7d1}, $ s \in S_{Y_m\vee Y'_n} \equiv Y_m\vee Y'_n \ge
x_{b_n}(s) $ for $ n>m $. We have to prove that
\[
\mathbb{P}\bigl(Y'_m=y \mbox{ \small and } \forall n >
m \> y' \vee Y'_n \ge x_{b_n}(s)
\bigr) = 0
\]
for every $ y \in b_m $ satisfying $ \leftidx{_y} K(s)=\infty$. By
Lemma \ref{7e7},
\[
\mathbb{P}\bigl(\forall n > m \> X_n \wedge y \wedge
x_{b_n}(s) = 0_\La\bigr) = 0.
\]
It remains to note that $ y' \vee Y'_n \ge x_{b_n}(s) \equiv( y
\wedge Y_n )' \ge x_{b_n}(s) \equiv( y \wedge Y_n ) \wedge x_{b_n}(s)
= 0_\La\imply y \wedge X_n \wedge x_{b_n}(s) = 0_\La$.
\end{pf}

Now we prove Proposition \ref{7e1}. We use \eqref{7e15} for $
b_1,b_2,\ldots$ and $ p_1,p_2,\ldots$ satisfying \eqref{7e2},
\eqref{7e3},
\[
\exists m\ \forall n\qquad s \in S_{Y_m\vee Y'_n}
\]
almost surely, for almost all $ s \in S $. In combination with
Lemma \ref{7e8} it gives
\[
\mathbb{P}\bigl(\exists m \> \leftidx{_{Y'_m}} K(s)<\infty\bigr) = 1
\]
for almost all $ s \in S $. On the other hand, by Lemma \ref{7e6} and
\eqref{7e2},
\[
\mathbb{P}\bigl(\exists m \> \leftidx{_{Y'_m}} K(s)<\infty\bigr) =
\lim
_m \mathbb{P}\bigl(\leftidx{_{Y'_m}} K(s)<\infty
\bigr) \le p_1 + p_2 + \cdots< 1
\]
for all $ s \in S $ such that $ K(s)=\infty$. Therefore $ K(s)<\infty
$ for almost all $ s $, which completes the proof of Propositions
\ref{7e1}, \ref{7a6} and finally, Theorem \ref{1c2}. Theorem~\ref{1c1}
follows immediately.

% imsref loaded by akundreckaite, 2013-08-28 15:13:01
%


%suskaldyti doi

\printaddresses

\end{document}